\documentclass[twoside]{amsart}
\usepackage{amssymb}
\usepackage{geometry}
\usepackage{graphicx}
\usepackage{esint}
\usepackage{esint}
\usepackage[dvipsnames]{xcolor}

\newtheorem{theorem}{Theorem}
\newtheorem{proposition}{Proposition}
\newtheorem{lemma}{Lemma}
\newtheorem{corollary}{Corollary}
\newtheorem*{lemma*}{Lemma}
\newtheorem{remark}{Remark}

\newcommand{\dd}{\,\mathrm{d}}
\newcommand{\dn}{\mathrm{d}}
\newcommand{\eps}{\varepsilon}
\newcommand{\R}{{\mathbb R}}
\newcommand{\Rnn}{{\mathbb{R}_{\ge0}}}
\newcommand{\Rp}{{\mathbb{R}_{>0}}}
\newcommand{\Rinf}{[0,\infty]}
\newcommand{\N}{\mathbb{N}}
\newcommand{\Z}{\mathbb{Z}}
\newcommand{\loc}{\text{loc}}
\newcommand{\Lloc}{L_\text{loc}}
\newcommand{\Wloc}{W_\text{loc}}

\newcommand{\prb}{\mathcal{M}^+_\mathrm{m}(\mathbb T)}
\newcommand{\prbr}[1]{L^1_\mathrm{m}(\mathbb T;#1)}
\newcommand{\nrg}{{\mathcal E}}

\newcommand{\fnc}{{\mathcal F}}
\newcommand{\mob}{{\mathbf m}}
\newcommand{\crc}{{\mathbb T}}
\newcommand{\intom}{\int_{\mathbb T}}

\newcommand{\id}{\mathrm{id}}
\newcommand{\htk}{{\mathcal G}}
\newcommand{\mk}{\widetilde{\mathcal G}}
\newcommand{\eins}{{\mathbf 1}}

\newcommand{\velo}{\mathbf{v}}
\newcommand{\welo}{\mathbf{w}}
\newcommand{\welost}{\mathbf{w}^s_t}
\newcommand{\rhost}{\rho^s_t}
\newcommand{\mobst}{{\mathbf m}(\rhost)}
\newcommand{\mobstp}{{\mathbf m}'(\rhost)}
\newcommand{\mobstpp}{{\mathbf m}''(\rhost)}
\newcommand{\mrho}{\bar\rho^\eps}
\newcommand{\mrhot}{\bar\rho^\eps_t}
\newcommand{\lrho}{\bar\rho^*}
\newcommand{\mwelo}{\bar{\mathbf{w}}^\eps}
\newcommand{\mwelot}{\bar{\mathbf{w}}^\eps_t}

\newcommand{\kinet}{{\mathcal K}}
\newcommand{\kinint}{\frac{\welost}{\mobst+\mu}}
\newcommand{\psit}{\psi_t}

\newcommand{\partialx}{\partial_x}
\newcommand{\partialxx}{\partial_{xx}}

\newcommand{\kfrac}[2]{\left[\frac{#1}{#2}\right]}

\newcommand{\wed}{{\mathcal W}}
\newcommand{\curves}{{\mathfrak C}}

\newcommand{\BL}{{\mathrm d}_\text{BL}}

\newcommand{\blG}{\mathbb G}
\newcommand{\opG}{\mathbf G}

\title[WED for thin film equations]{Solutions to fourth order degenerate parabolic equations obtained via weighted energy dissipation}
\author{Daniel Matthes}
\author{Vera Pazukhina}

\begin{document}

\begin{abstract}
  We apply the variational method of Weighted Energy Dissipation (WED)
  to obtain a global-in-time approximation of solutions
  to fourth order degenerate parabolic equations of Cahn-Hilliard type.
  Differently from the standard approach to the existence theory,
  WED induces an elliptic regularization in time, not in space.
  The confinement on the phase field is guaranteed already for the approximation,
  a priori estimates follow without modification of Lyapunov functionals,
  and the approximations are weakly differentiable in time
  --- even twice in the case of linear mobility.
  While WED has been widely used for the construction of gradient flows in Hilbert spaces,
  this appears to be the first application of WED to a metric gradient flow
  beyond second order PDEs of Wasserstein type.
\end{abstract}

\maketitle

%%%%%%%%%%%%%%%%%%%%%%%%%%%%%%%%%%%%%%%%%%%%%%%%%%%%%%%
%%%%%%%%%%%%%%%%%%%%%%%%%%%%%%%%%%%%%%%%%%%%%%%%%%%%%%%
\section{Introduction}
%%%%%%%%%%%%%%%%%%%%%%%%%%%%%%%%%%%%%%%%%%%%%%%%%%%%%%%
%%%%%%%%%%%%%%%%%%%%%%%%%%%%%%%%%%%%%%%%%%%%%%%%%%%%%%%
%
% In this paper, we study the elliptic regularization in time by the WED approximation
% for two evolution equations of gradient flow type:
% the thin film equation with linear mobility for the thickness $\rho(t;x)\ge0$ of a thin viscous film over a flat substrate,
% \begin{align}
%   \label{eq:intro-tf}
%   \partial_t\rho = -\partial_x(\rho\,\partial_{xxx}\rho),
% \end{align}
% and the Cahn-Hilliard equation for the phase field parameter $\rho(t;x)\in[0,1]$,
% \begin{align}
%   \label{eq:intro-ch}
%   \partial_t\rho = -\partial_x\big(\mob(\rho)\,\partial_{xxx}\rho\big),
% \end{align}
% with a degenerate uniformly concave mobility $\mob:[0,1]\to\Rp$, i.e., $\mob(0)=\mob(1)=0$ and $\mob''\le-1$.
% We construct differentiable-in-time non-negative global approximations of solutions to \eqref{eq:intro-tf} and \eqref{eq:intro-ch}
% without regularization in space, and we prove convegence to a weak solution in the temporal inviscid limit.
% Our approch is complementary to the one established for abstract metric gradient flows \cite{AGS,RSSS};
% we use elements of that theory, but characterize the regularized solutions directly in terms of PDEs.

%%%%%%%%%%%%%%%%%%%%%%%%%%%%%%%%%%%%%%%%%%%%%%%%%%%%%%%
\subsection{The PDE}
%%%%%%%%%%%%%%%%%%%%%%%%%%%%%%%%%%%%%%%%%%%%%%%%%%%%%%%
%
In this paper, we apply the variational method of \emph{Weighted Energy Dissipation (WED)}
% --- also called parabolic \emph{Weighted Inertia-Dissipation-Energy (WIDE) principle} ---
to approximate non-negative weak solutions $\rho$ to certain degenerate fourth order parabolic PDEs globally in time.
For definiteness (and simplicity), we consider
\begin{align}
  \label{eq:intro-ch}
  \partial_t\rho + \partial_x\big(\mob(\rho)\,\partial_{xxx}\rho\big) = 0
\end{align}
on the one-dimensional torus $\crc$.
In \eqref{eq:intro-ch}, the mobility $\mob:I\to\Rnn$ is a prescribed concave function over the interval $I\subseteq[0,\infty)$ of admissible values for $\rho$,
which falls into one of the following two categories:
\begin{itemize}
\item \emph{Thin film:}
  $\mob(\rho)=\rho$ is linear on $I=[0,\infty)$.
  In this case, \eqref{eq:intro-ch} is a special thin film equation,
  that models the thickness $\rho(t;x)\ge0$ of a thin viscous film over a flat substrate, moving by surface tension.
\item \emph{Cahn-Hilliard:}
  $\mob$ is a Lipschitz continuous uniformly concave function on $I=[0,1]$, with $\mob(0)=\mob(1)=0$.
  In this case, \eqref{eq:intro-ch} is a degenerate Cahn-Hilliard equation for the phase field $\rho(t;x)\in[0,1]$.
\end{itemize}
Existence of weak solutions to \eqref{eq:intro-ch} in these --- and far more general --- cases are classical results,
see e.g. \cite{BF,Gruen} for the thin film equation and \cite{EG} for the Cahn-Hilliard equation.
The typical strategy for proving existence is via regularization of the mobility,
making the problem genuinely parabolic.
The approximations are smooth, but the price to pay is a violation of the constraints $\rho\ge0$ or $\rho\in[0,1]$, respectively,
and the necessity to modify the relevant Lyapunov functionals to obtain uniform a priori estimates.
A complementary approach is via a variational time-discretization
--- see \cite{GO} for a result in the thin film case and \cite{LMS} for the Cahn-Hilliard case ---
in which the constraint is built into the minimization and Lyapunov functionals are inherited by convexity,
but in exchange for an approximation that is discontinuous in time.
Further approaches (all with their pros and cons) to existence, also for more general thin film equations, include
the representation as an obstacle problem \cite{Gobstacle},
the discretization in space \cite{GR},
and the reformulation in terms of Lagrangian maps \cite{Naldi,Osberger}.

To the best of our knowledge,
this is the first time that WED is applied to \eqref{eq:intro-ch}.
The method is fully variational and produces global time-continuous approximations
that preserve mass and the constraint $\rho\in I$ by construction.
The central a priori estimate --- on $\partial_{xx}\rho$ in $L^2$ --- can be obtained by variation along the heat flow.
The approximations are at least once weakly differentiable in time, and even twice in the thin film case.
Apart from the preliminary results on the very particular fourth order DLSS equation in the thesis \cite{Simon},
this is apparently the first application of the WED method
to a specific gradient flow outside of the Hilbertian setting and beyond second order drift-diffusion.

%%%%%%%%%%%%%%%%%%%%%%%%%%%%%%%%%%%%%%%%%%%%%%%%%%%%%%%
\subsection{WED regularization of gradient flows}
%%%%%%%%%%%%%%%%%%%%%%%%%%%%%%%%%%%%%%%%%%%%%%%%%%%%%%%
%
The idea to regularize evolution equations by introduction of an additional second order time derivative dates back at least to the 1960's.
The inclusion of that regularizing term in a variational form is more recent \cite{Ilmanen},
and so are the applications of the WED principle to gradient flows \cite{ContiOrtiz,MielkeOrtiz}.
The interested reader is referred to the comprehensive survey \cite{Uli} for a historical account,
and also for an overview over the countless current applications of the variational elliptic-in-time regularization,
that go far beyond gradient flows and include, for instance, image analysis and the Navier-Stokes equations.

We briefly illustrate the general idea of WED.
For the gradient flow $\dot x=-\nabla E(x)$ of a smooth bounded energy function $E$ on Euclidean space,
the elliptic-in-time regularization with parameter $\eps>0$ is given by
\begin{align}
  \label{eq:ETR-Rd}
  -\eps\ddot x+\dot x=-\nabla E(x).
\end{align}
A global-in-time solution $x_*:[0,\infty)\to\R^d$ to \eqref{eq:ETR-Rd} coincides 
with the unique minimizer $\gamma_*$ of the so-called WED functional,
\begin{align*}
  \wed^\eps(\gamma) = \int_0^\infty \frac{e^{-t/\eps}}{\eps}
  \left[\frac\eps 2\big|\dot\gamma(t)\big|^2 + E\big(\gamma(t)\big)\right]\dd t,
\end{align*}
which is defined on all $H^1$-regular curves $\gamma:[0,\infty)\to\R^d$ with fixed initial datum $\gamma(0)=x(0)$.

On a Riemannian manifold with metric tensor $g$,
the gradient flow equation for $E$ can be written
in local coordinates as $\dot x^k=-\sum_\ell g^{k\ell}(x)\partial_\ell E(x)$.
Minimizers of the corresponding WED functional
\begin{align}
  \label{eq:WED-mf}
  \wed^\eps(\gamma) = \int_0^\infty \frac{e^{-t/\eps}}{\eps}
  \left[\frac\eps 2\sum_{k,\ell} g_{k\ell}(\gamma(t))\dot\gamma^k(t) \dot\gamma^\ell(t) + E\big(\gamma(t)\big)\right]\dd t
  \end{align}
satisfy a variant of \eqref{eq:ETR-Rd}, with the flat acceleration replaced by the covariant derivative,
\begin{align}
  \label{eq:ETR-mf}
  -\eps\left(\ddot x^k + \sum_{i,j}\Gamma^k_{ij}(x)\dot x^i\dot x^j\right)
  + \dot x^k = - \sum_\ell g^{k\ell}(x)\partial_\ell E(x),
\end{align}
where $\Gamma$ is $g$'s Christoffel symbol.
In the smooth finite-dimensional setting above, i.e., both in the flat euclidean or the curved manifold case,
it is elementary to show that the regularized solutions converge to the gradient flow as $\eps\searrow0$.

Naturally, the significance of the WED approach lies in the construction of solutions
to gradient flows in certain \emph{non-smooth infinite-dimensional} generalizations of the situations above.
With its global variational nature, WED opens an angle of attack
to study existence and regularity of solutions in abstract evolution problems,
as well as $\Gamma$-limits \cite{Liero} and optimal control \cite{Fukao}.
Applications to specific PDEs, see e.g. \cite{Audrito1,Audrito2,Duzaar,Bernd,Ulisse},
are so far mostly set in Hilbert spaces,
and thus are infinite dimensional versions of the flat euclidean flow \eqref{eq:ETR-Rd}.
A general approach to the construction of Hilbertian gradient flows
via weighted energy dissipation has been developed in \cite{MielkeStefanelli}.
The natural infinite-dimensional version of the flow \eqref{eq:ETR-mf} on Riemannian manifolds
is the WED regularization of \emph{metric gradient flows}. % see \cite{AGS} for a comprehensive treatment.
A metric counterpart to the Hilbertian approach in \cite{MielkeStefanelli} has been developed recently \cite{RSSS}:
there, it has been shown that the inviscid limit of the WED approximation produces curves of maximal slope
under the same general conditions that also guarantee convergence of the time-discrete minimizing movement scheme,
see \cite{AGS}.
Applications include gradient flows in Banach spaces or in the $L^2$-Wasserstein metric,
but were so far limited to second order drift-diffusion equations.

%%%%%%%%%%%%%%%%%%%%%%%%%%%%%%%%%%%%%%
\subsection{Motivation for this work}
%%%%%%%%%%%%%%%%%%%%%%%%%%%%%%%%%%%%%%
%
The degenerate fourth order PDEs \eqref{eq:intro-ch} fall into the framework of metric gradient flows.
Since the mobility function $\mob$ is assumed to be concave, there is an associated (generalized) transport metric \cite{DNS,LisMar};
for the linear mobility, that metric is the celebrated $L^2$-Wasserstein distance.
Thus the abstract approach from \cite{RSSS} guarantees the convergence of elliptic-in-time regularized solutions of a curve of maximal slope.

There are mainly three reasons why we pursue a different, more hands-on approach in the paper at hand.
The first and foremost is that \cite{RSSS}
is focused on the recovery of a metric gradient flow in the limit $\eps\searrow0$,
and WED minimizers are studied almost exclusively under this perspective;
some finer properties are discussed in the semi-convex case,
but solely on the abstract level of metric derivatives and subdifferentials.
Here, we provide a quite explicit characterization of the evolution equation satisfied by the WED minimizers:
\eqref{eq:tffinal} and \eqref{eq:chfinal} give a rigorous mathematical meaning
to the differential geometric formulation \eqref{eq:ETR-mf} above in the context of \eqref{eq:intro-ch}.
In particular, the Hamilton-Jacobi terms that govern acceleration in the metric space become clearly visible,
and for $\mob(\rho)=\rho$, second order weak differentiability of $\rho$ in time is obtained. 

The second reason is that for more complicated PDEs like \eqref{eq:intro-ch},
the abstract metric theory is not completely satisfactory:
while the existence and approximation of a curve of maximal slope is guaranteed under very general conditions,
the evolution equation satisfied by that curve needs to be identified
% in a auxiliary step that is very specific to the PDE under consideration, namely
by establishing a chain rule, which is very specific to the problem at hand.
% or equivalently by explicitly characterizing the metric slope in functional-analytic terms.
The proof of chain rules can be extremely cumbersome already for second order equations, see \cite{Caillet,Gess}.
To the best of our knowledge, the only chain rule that has ever been rigorously proven for a fourth order PDE
is for the Quantum Drift Diffusion model \cite{GST} --- see also \cite[Example 11.1.10]{AGS} ---
which is the $L^2$-Wasserstein gradient flow of the Fisher information functional.
It is not evident how to transfer that result to the Dirichlet energy, and to the case of nonlinear mobilities $\mob$.

The third reason is that the WED approximation provides a variational tool
to construct also ``non-genuine'' metric gradient flows,
that is PDEs with dissipative structure, but without associated metric.
The primary class of examples are of the form \eqref{eq:intro-ch} with \emph{non-concave} mobilities
like $\mob(r)=r^m$ for $r\ge0$, or $\mob(r)=r^m(1-r)^m$ for $r\in[0,1]$, with some $m>1$.
The approximation of such generalized metric gradient flows by WED
is the aim of a forthcoming paper \cite{Parsch}, which builds on --- but vastly extends --- the methods developed here.

%%%%%%%%%%%%%%%%%%%%%%%%%%%%%%%%%%%%%%%%%%%%%%%%%%%%%%%
\subsection{The PDE as gradient flow}
%%%%%%%%%%%%%%%%%%%%%%%%%%%%%%%%%%%%%%%%%%%%%%%%%%%%%%% 
%
Following the formal geometric concept developed by Otto in \cite{OttoPME} for the $L^2$-Wasserstein metric,
we consider \eqref{eq:intro-ch} as gradient flows of the Dirichlet functional
\begin{align}
  \label{eq:intro-Dirichlet}
  \nrg(\rho) = \frac12\intom \big(\partial_x\rho\big)^2\dd x
\end{align}
on the infinite dimensional Riemannian manifold of densities $\prbr{I}$ over the one-dimensional torus $\crc$
with values in $I$ and total mass $\mathrm m>0$, with $\mathrm m<1$ if $I=[0,1]$.
The Riemannian structure on $\prbr{I}$ is defined as follows:
for a tangent vector at $\rho$, i.e., a function $\dot\rho\in L^1(\crc)$ of vanishing integral,
define its squared norm as the minimal kinetic energy needed to realize $\dot\rho$ with given mobility $\mob(\rho)$,
\begin{align*}
  \|\dot\rho\|_\rho^2 := \inf\left\{\intom\frac{\welo^2}{\mob(\rho)}\dd x\,\middle|\,\dot\rho+\partial_x\welo=0\right\}.
\end{align*}
For sufficiently regular $\rho$ and $\dot\rho$, the above infimum is attained.
%and the corresponding optimal $\welo$ is of gradient type,
%\begin{align}
%  \label{eq:welograd}
%  \welo = \mob(\rho)\partial_x\varphi,
%  \quad \text{that is},\quad
%  \dot\rho = -\partial_x\big(\mob(\rho)\,\partial_x\varphi).
%\end{align}
% Moreover,
The norm polarizes and thus gives rise to a scalar product on tangent vectors.
Proceeding further formally by computing the Christoffel symbols and the associated covariant derivative,
one obtains the following PDE
as an analogue of the elliptic-in-time regularization \eqref{eq:ETR-mf} on Riemannian manifolds:
\begin{align}
  \label{eq:ETR-intro}
  \begin{split}
    \eps\,\partial_x\left(\mob(\rho)\,\left\{\partial_t\left[\frac{\welo}{\mob(\rho)}\right]
        +\partial_x\left[\frac{\mob'(\rho)}2\left(\frac{\welo}{\mob(\rho)}\right)^2\right]\right\}\right)
    % \eps\partial_x\left(\mob(\rho)\,\partial_x\left\{\partial_t\varphi+\frac{\mob'(\rho)}2\big(\partial_x\varphi\big)^2\right\}\right)
    % \left\{\partial_t\left(\frac{\welo}{\mob(\rho)}\right)+\partial_x\left(\frac{\mob'(\rho)}2\left(\frac{\welo}{\mob(\rho)}\right)^2\right)\right\}
    +\partial_t\rho
    = -\partial_x\big(\mob(\rho)\,\partial_{xxx}\rho\big),
  \end{split}
\end{align}
where $\rho$ and $\welo$ are coupled by the continuity equation % \eqref{eq:welograd},
\begin{align}
  \label{eq:strongcont}
  \partial_t\rho + \partial_x\welo = 0.
\end{align}
The associated WED functional, in analogy to \eqref{eq:WED-mf}, is given by
\begin{align}
  \label{eq:WED-intro}
  \wed^\eps(\rho,\welo)
  = \int_0^\infty \frac{e^{-t/\eps}}{\eps}\left[
  \frac\eps2\intom \frac{\welo_t^2}{\mob(\rho_t)}\dd x
  + \nrg(\rho_t)
  \right] \dd t,
\end{align}
where $\rho=(\rho_t)_{t\ge0}$ is a curve of densities $\rho_t\in\prbr{I}$,
and $\welo=(\welo_t)_{t\ge0}$ is a curve of integrable functions $\welo_t\in L^1(\crc)$,
that are connected to one another by means of the continuity equation \eqref{eq:strongcont}.
% The goal of this paper is prove the existence minimizers of \eqref{eq:WED-intro},
% study their relation to the PDE \eqref{eq:ETR-intro},
% and to prove their convergence to weak solutions of \eqref{eq:intro-ch}.

%%%%%%%%%%%%%%%%%%%%%%%%%%%%%%%%%%%%%%%%%%%%%%%%%%%%%%%
\subsection{Main results}
%%%%%%%%%%%%%%%%%%%%%%%%%%%%%%%%%%%%%%%%%%%%%%%%%%%%%%%
%
We begin by summarizing our results for the thin film case,
which are slightly stronger than for the Cahn-Hilliard one further below.
A reason is that the regularizing term in the PDE \eqref{eq:ETR-intro} simplifies significantly for $\mob'\equiv1$,
and thus admits to prove the following:
\begin{theorem}
  \label{thm:tfpde}
  Assume $\mob(r)=r$, and let $\bar\rho_0\in\prbr{\Rnn}$ with $\nrg(\bar\rho_0)<\infty$.
  For each $\eps>0$, the functional $\wed^\eps$ possesses
  a minimizer $(\mrho,\mwelo)$ of regularity
  \[ \mrho\in L^2_\loc\big((0,\infty);H^2(\crc)\big)\cap W^{2,1}_\loc\big((0,\infty);W^{2,\infty}(\crc)^*\big),
    \quad \frac{(\mwelo)^2}{\mrho} \in L^1_\loc\big((0,\infty)\times\crc\big),\]
  that is weakly continuous in time with $\mrho_0=\bar\rho_0$
  and satisfies the PDE \eqref{eq:ETR-intro} in distributional sense:
  \begin{align}
    \label{eq:tffinal}
    \eps\left[-\partial_{tt}\mrho+\partial_{xx}\left(\frac{\big(\mwelo\big)^2}{\mrho}\right)\right]
    + \partial_t\mrho
    = -\partial_{xx}\left(\mrho\,\partial_{xx}\mrho - \frac12\big(\partial_x\mrho\big)^2\right).
  \end{align}
\end{theorem}
About the inviscid-in-time limit of the WED approximation, we show:
\begin{theorem}
  \label{thm:tflimit}
  Any sequence of minimizers $(\mrho,\mwelo)$ for $\wed^\eps$ with $\eps\searrow0$
  contains a (non-relabeled) subsequence such that 
  \begin{align}
    \label{eq:tflimit9}
    \mrho \to \lrho \quad \text{strongly in $L^2_\loc\big((0,\infty);H^1(\crc)\big)$ and weakly in $L^2_\loc\big((0,\infty);H^2(\crc)\big)$},
  \end{align}
  with a limit $\lrho:[0,\infty)\to\prbr{\Rnn}$
  that is weakly continuous in time with $\lrho(0)=\bar\rho_0$
  and satisfies the following weak form of the thin film equation \eqref{eq:intro-ch}:
  \begin{align}
    \label{eq:tf0weak}
    0 = \int_0^\infty\intom \left(\lrho\,\partial_t\Phi+\frac32\big(\partial_x\lrho\big)^2\,\partial_{xx}\Phi+\lrho\,\partial_x\lrho\,\partial_{xxx}\Phi\right)\dd x\dd t
  \end{align}
  for any test function $\Phi\in C^\infty_c\big((0,\infty)\times\crc\big)$.
\end{theorem}
We are unable to prove a direct analogue of Theorem \ref{thm:tfpde}
for the Cahn-Hilliard equation \eqref{eq:intro-ch} with strictly concave $\mob$;
the main reason is that the a priori estimates are too weak to 
give a meaning to the quotient $\welo^2/\mob(\rho)$.
We shall therefore use a modified WED functional $\widehat\wed^\eps$
in which the denominator of the kinetic term $\welo^2/\mob(\rho)$ is replaced by the uniformly positive expression $\mob(\rho)+\mu$,
with a constant $\mu$ that tends to zero with $\eps$ as $\eps=o(\mu^3)$.
Formally, this appears to be a parabolic regularization of the problem,
but is actually quite different in the analysis, because the constraint $\rho\in I=[0,1]$ is still enforced,
which potentially reduces the smoothness of $\mrho$ near points where $\mrho=0$ or $\mrho=1$.
\begin{theorem}
  \label{thm:chpde}
  Let $\mob:[0,1]\to\R$ be uniformly concave and Lipschitz with $\mob(0)=\mob(1)=0$ and $-\mob''\ge1$. Let $\bar\rho_0\in\prbr{[0,1]}$ with $\nrg(\bar\rho_0)<\infty$.
  For each $\eps>0$, the functional $\widehat\wed^{\eps}$ possesses
  a minimizer $(\mrho,\mwelo)$ of regularity
  \[ \mrho\in L^2_\loc\big((2\eps,\infty);H^2(\crc)\big),\quad \partial_t\mrho,\,\mwelo\in L^2_\loc\big((2\eps,\infty)\times\crc\big), \]
  that is weakly continuous in time with $\mrho_0=\bar\rho_0$
  and satisfies the following perturbed version of the PDE \eqref{eq:ETR-intro}: 
  \begin{equation}
    \label{eq:chfinal}
    \begin{split}
     0=&\int_0^\infty\intom \Big\{\partial_{xx}\mrho\,\partial_x\big(\mob(\mrho)\xi\big) + \frac{\mob(\mrho)}{\mob(\mrho)+\mu}\mwelo\xi\Big\} \dd x\dd t \\
      &+ \eps \int_0^\infty\intom\left\{\frac{\mwelo}{\mob(\mrho)+\mu}\partial_t\big(\mob(\mrho)\xi\big)
        + \frac{\mob'(\mrho)}2\left(\frac{\mwelo}{\mob(\mrho)+\mu}\right)^2\partial_x\big(\mob(\mrho)\xi\big)\right\}\dd x\dd t
    \end{split}
  \end{equation}
  for every test function $\xi\in C^\infty_c\big((2\eps,\infty)\times\crc\big)$.
\end{theorem}
In contrast to Theorem \ref{thm:tfpde},
we are not able to obtain a weak second time derivative of $\rho$,
because of the incongruence between the strictly positive denominator $\mob(\rho)+\mu$
and the degenerate prefactor $\mob(\rho)$ of the test function $\xi$.
The closest feasible result would be a weak first time derivative of the quotient $\mob(\rho)\welo/(\mob(\rho)+\mu)$. 
Still, the passage to the limit $\eps\to0$ works as above:
\begin{theorem}
  \label{thm:chlimit}
  Any sequence with $\eps\searrow0$ contains a (non-relabeled) sub-sequence such that 
  \begin{align*}
    \mrho \to \lrho \quad \text{strongly in $L^2_\loc\big((0,\infty);H^1(\crc)\big)$ and weakly in $L^2_\loc\big((0,\infty);H^2(\crc)\big)$},
  \end{align*}
  with a limit $\lrho:[0,\infty)\to\prbr{[0,1]}$ that is weakly continuous in time with $\lrho(0)=\bar\rho_0$
  and satisfies the following weak form of \eqref{eq:intro-ch}:
  \begin{align}
    \label{eq:tf0weak_m}
    0 = \int_0^\infty\intom \big(-\lrho\,\partial_t\Phi+\partial_{xx}\lrho \,\partial_x(\mob(\lrho)\,\partial_x\Phi)\big)\dd x\dd t
  \end{align}
  for any test function $\Phi\in C^\infty_c\big((0,\infty)\times\crc\big)$.
\end{theorem}

%%%%%%%%%%%%%%%%%%%%%%%%%%%%%%%%%%%%%%%%%%%%%%%%%%%%%%%
\subsection{Plan of the paper}
%%%%%%%%%%%%%%%%%%%%%%%%%%%%%%%%%%%%%%%%%%%%%%%%%%%%%%%
%
In Section \ref{sct:WED}, we prove existence of minimizers
of the WED functional in the specific context at hand,
and derive several general a priori estimates in analogy to \cite{RSSS},
like $\eps$-uniform integrability of kinetic and potential energy.
In Section \ref{sct:tf}, we specialize to the thin film case $\mob(\rho)=\rho$:
we obtain Euler-Lagrange equations and regularity estimates for the WED minimizers,
and finally pass to the limit $\eps\searrow0$.
In Section \ref{sct:ch}, the analogous program is carried out
in the Cahn-Hilliard case with uniformly concave $\mob$;
this is much more challenging on the technical level.

%%%%%%%%%%%%%%%%%%%%%%%%%%%%%%%%%%%%%%%%%%%%%%%%%%%%%%%
\subsection{Notation}
%%%%%%%%%%%%%%%%%%%%%%%%%%%%%%%%%%%%%%%%%%%%%%%%%%%%%%%
%
\begin{center}
\begin{tabular}{l l}
 $\crc$ & one-dimensional torus\\
 $\prb$ & non-negative Borel measures on $\crc$ of total mass $\mathrm{m}$ \\ 
 $\prbr{I}$ & absolutely continuous measures from $\prb$ with values in $I$\\
 $\BL$ &  bounded Lipschitz distance on $\prb$
\end{tabular}
\end{center}
% More precisely,
% \begin{align*}
%   \prbr{I}=\left\{ \rho\in L^1(\crc)\,\middle|\, \rho(x)\in I \text{ for almost every } x\in\crc, \intom\rho(x)\dd x=\mathrm{m}\right\}
% \end{align*}
% where $I=\Rnn$ in case $\mob(\rho)=\rho$ or $I=[0,1]$ corresponding to a more general concave mobility $\mob$.
% 
$\BL$ is defined for any $\mu_0,\mu_1\in\prb$ as
\begin{align*}
  \BL(\mu_0,\mu_1) = \sup\left\{\intom f\,\big(\mu_1-\mu_2\big)\dd x \,
  \middle|\, f\in C^{0,1}(\crc),\ |f(x')-f(x)|\le|x'-x|\ \text{for all $x,x'\in\crc$}\right\}.
\end{align*}
The bounded Lipschitz distance is a metric on $\prb$,
and the metric space $(\prb,\BL)$ is complete.
% For absolutely continuous $\rho_0,\rho_1\in\prbr{I}$ it is understood that
% \begin{align*}
%   \BL(\rho_0,\rho_1) = \BL(\rho_0\Leb,\rho_1\Leb). 
% \end{align*}

%%%%%%%%%%%%%%%%%%%%%%%%%%%%%%%%%%%%%%%%%%%%%%%%%%%%%%%
%%%%%%%%%%%%%%%%%%%%%%%%%%%%%%%%%%%%%%%%%%%%%%%%%%%%%%%
\section{The WED method for diffusion equations}
\label{sct:WED}
%%%%%%%%%%%%%%%%%%%%%%%%%%%%%%%%%%%%%%%%%%%%%%%%%%%%%%%
%%%%%%%%%%%%%%%%%%%%%%%%%%%%%%%%%%%%%%%%%%%%%%%%%%%%%%%
%
Throughout this section, the mobility function $\mob$ is either
defined on $I=[0,\infty)$ and of the simple linear form $\mob(r)=r$, or
defined on $I=[0,1]$, non-negative and strictly concave;
we \emph{do not} require degeneracy $\mob(0)=\mob(1)=0$ in this case.
This generalization will be important when we need to regularize the mobility
in the form $\mob+\mu$ with $\mu>0$ in Section \ref{sct:ch}.
In any case, the densities $\rho\in L^1(\crc)$ under consideration
are always subject to the condition $\rho(x)\in I$.

%%%%%%%%%%%%%%%%%%%%%%%%%%%%%%%%%%%%%%%%%%%%%%%%%%%%%%%
\subsection{WED functional}
%%%%%%%%%%%%%%%%%%%%%%%%%%%%%%%%%%%%%%%%%%%%%%%%%%%%%%%
%
Define the \emph{kinetic energy density} $\kinet:\prbr{I}\times L^1(\crc)\to\Rinf$ by
\begin{align*}
  \kinet(\rho,\welo)
  =\frac{1}{2} \intom \kfrac{\welo^2}{\mob(\rho)}\dd x,
  \quad\text{where}\quad
  \kfrac{w^2}{\mob(r)}
  =
  \begin{cases}
    \frac{w^2}{\mob(r)} & \text{if $\mob(r)>0$}, \\
    0 & \text{if $\mob(r)=0$ and $w=0$}, \\
    +\infty & \text{otherwise}.
    % \text{if $r=0$ and $w\neq0$}.
  \end{cases}
\end{align*}
\begin{remark}
\label{rmk:0by0}
The definition of $\kfrac{\cdot}{\cdot}:\Rnn\times\Rnn\to\Rinf$ is made such that
the expression $\kfrac{w^2}{\mob(r)}$ is lower semi-continuous and jointly convex in $(r,w)$;
here the concavity of $\mob$ is essential.
By definition, $\kinet(\rho,\welo)$ is non-negative, but possibly $+\infty$.
Finiteness $\kinet(\rho,\welo)<\infty$ implies
that $\rho\in I$ a.e.\ on $\crc$, and that $\welo=0$ a.e.\ on $\{\mob(\rho) = 0\}$.
We remark that there exists a lower semi-continuous extension of $\kinet$ to pairs $(\mu,\omega)$
of general Borel measures $\mu$ and Radon measures $\omega$ on $\crc$;
this is the starting point for the metric theory in \cite{DNS}.
\end{remark}
The rigorous definition of the energy $\nrg:\prbr{I}\to\Rinf$ from \eqref{eq:intro-Dirichlet} is
\begin{align*}
  \nrg(\rho) =
  \begin{cases}
    \frac12\intom \big(\partial_x\rho\big)^2\dd x & \text{if $\rho\in H^1(\crc)$}, \\
    +\infty & \text{otherwise}.
  \end{cases}
\end{align*}
This functional is convex and lower semi-continuous with respect to weak convergence.

For definition of the WED functional, we first need to introduce an appropriate class of admissible curves.
Specifically, for a given initial datum $\bar\rho_0\in\prbr{I}$ with $\nrg(\bar\rho_0)<\infty$,
we consider pairs $(\rho,\welo)$ of measurable maps $\Rnn\ni t\mapsto(\rho(t),\welo(t))\in L^1(\crc)\times L^1(\crc)$
with the following properties:
\begin{enumerate}
\item[(C1)] $\rho(t)\in\prbr{I}$ for each $t\ge0$,
\item[(C2)] $\Rnn\ni t\to\rho(t)$ is continuous in $\BL$, with $\rho(0)=\bar\rho_0$,
\item[(C3)] $\Rnn\ni t\to\welo(t)$ is locally integrable, and
\item[(C4)] the continuity equation \eqref{eq:strongcont} holds in the sense of distributions, 
i.e., for all $\theta\in C^1_c((0,\infty)\times\crc)$:
  \begin{align}
    \label{eq:weakcont}
%    \intom \theta(t_2;x)\rho(t_2;x)\dd x - \intom \theta(t_1;x)\rho(t_1;x)\dd x
    0
    = \int_0^\infty\intom\big(\partial_t\theta\,\rho+\partial_x\theta\,\welo\big)\dd x\dd t.
  \end{align}
  % with arbitrary $0\le t_1<t_2<\infty$.
\end{enumerate}
Accordingly, the class of admissible curves is defined by
\begin{align*}
  \curves(\bar\rho_0) = \left\{(\rho,\welo):\Rnn\to L^1(\crc)\times L^1(\crc)
  \,\middle|\,
  \text{$\rho$ and $\welo$ satisfy (C1) to (C4)} \right\},
\end{align*}
and on $\curves(\bar\rho_0)$, we define the WED functional $\wed^\eps:\curves(\bar\rho_0)\to\Rinf$ by
\begin{align*}
  \wed^\eps(\rho,\welo)
  = \int_0^\infty \big[\eps\kinet\big(\rho(t),\welo(t)\big) + \nrg\big(\rho(t)\big)\big] \dd\mu^\eps(t)
\end{align*}
where $\mu^\eps$ is a Borel measure on $(0, \infty)$ with density $\frac{e^{-t/\eps}}{\eps}$.

%%%%%%%%%%%%%%%%%%%%%%%%%%%%%%%%%%%%%%%%%%%%%%%%%%%%%%%
\subsection{Existence of minimizers}
%%%%%%%%%%%%%%%%%%%%%%%%%%%%%%%%%%%%%%%%%%%%%%%%%%%%%%%
%
\begin{proposition}
  \label{prp:x1}
  For any initial datum $\bar\rho_0\in \prbr{I}\cap H^1(\crc)$ and any $\eps>0$,
  there exists a minimizer $(\mrho,\mwelo)\in\curves(\bar\rho_0)$ of $\wed^\eps$.
\end{proposition}
\begin{remark}
  Existence of minimizers could alternatively be concluded
  by means of the abstract machinery from \cite{RSSS} from general properties of $\nrg$
  in combination with the characterization of transport metric with concave mobility $\mob$ \cite{DNS}.
  We give a more elementary proof here, which makes no use of the metric structure.
\end{remark}
\begin{proof}
  First note that the functional $\wed^\eps$ has a finite infimum:
  as integral over a non-negative quantity, it is bounded from below,
  and it is also proper, since the trivial curve $(\rho,\welo)\equiv(\bar\rho_0,0)$
  belongs to $\curves(\bar\rho_0)$ and has value $\wed^\eps(\bar\rho_0,0)=\nrg(\bar\rho_0)$,
  which is finite by hypothesis.
  Let $(\rho_n,\welo_n)$ be a minimizing sequence, and set
  \begin{align*}
    \underline\nrg := \lim_{n\to\infty}\wed^\eps(\rho_n,\welo_n) \ge0,
    \quad
    \overline\nrg := \sup_{n}\wed^\eps(\rho_n,\welo_n) < \infty.
  \end{align*}
  The goal of the rest of the proof is to show the existence
  of some curve $(\bar\rho^\eps,\bar\welo^\eps)\in\curves(\bar\rho_0)$
  --- appearing as limit of $(\rho_n,\welo_n)$ is a suitable sense ---
  which satisfies
  \begin{align}
    \label{eq:xgoal}
    \wed^\eps(\bar\rho^\eps,\bar\welo^\eps) \le \underline\nrg.
  \end{align}

  \emph{Construction of $\bar\rho^\eps$.}
  We shall prove that there is a subsequence $(\rho_{n_k})_k$
  and a $\BL$-continuous locally uniform limit curve $\bar\rho^\eps:\Rnn\to\prbr{I}$, i.e.,
  \begin{align}
    \label{eq:xBLconv}
    \lim_{n_k\to\infty}\sup_{t\in[0,T]}\BL\big(\rho_{n_k}(t),\bar\rho^\eps(t)\big) = 0
      \quad \text{for each $T>0$}.
  \end{align}
  For that curve $\bar\rho^\eps$, it will further follow that 
  \begin{align}
    \label{eq:epsL1conv}
    &\rho_{n_k}\to\bar\rho^\eps \quad\text{in $L^1\big((0,T)\times\crc\big)$, for each $T>0$, and} \\
    \label{eq:rhoconv2}
    &\int_0^\infty \nrg\big(\bar\rho^\eps(t)\big)\dd\mu^\eps(t)
      \le\liminf_{n_k\to\infty}\int_0^\infty \nrg\big(\rho_{n_k}(t)\big)\dd\mu^\eps(t).
  \end{align}
  The key to the above is to show that the curves $\Rnn\ni t\mapsto\rho_n(t)$
  are $n$-uniformly locally H\"older continuous in $\BL$,
  \begin{align}
    \label{eq:xBLhoelder}
    \BL\big(\rho_n(t_1),\rho_n(t_0)\big) \le C_T|t_1-t_0|^{1/2} \quad \text{for all $t_0,t_1\in[0,T]$},
  \end{align}
  with a constant $C_T$ just depending on $T>0$.
  Observe that
  \begin{align}
    \label{eq:wL1}
    \intom|\welo_n(t)|\dd x \le \Big[2\kinet\big(\rho_n(t),\welo_n(t)\big)\Big]^{1/2}\left[\intom \mob(\rho_n(t)) \dd x\right]^{1/2}\le \Big[2M\,\kinet\big(\rho_n(t),\welo_n(t)\big)\Big]^{1/2}
  \end{align}
  for some $M>0$ uniform in $n$:
  one may take $M=\mathrm{m}$ in the linear case $\mob(\rho)=\rho$,
  and $M=\|\mob\|_\infty^{1/2}$ in the bounded case.
  % Here we have used H\"older's inequality and definition of $\kinet$.
    Thanks to the above integrability of $\welo_n$, and since $\rho_n$ is weakly continuous with respect to time, we can follow standard arguments --- see e.g. \cite[Chapter 8]{AGS} ---
    to upgrade from the distributional form \eqref{eq:weakcont} of the continuity equation to a stronger form that admits in particular $C^{1,0}$-test functions over a compact time interval $[t_0,t_1]\subset[0,T]$ that do not necessarily vanish at the initial and terminal point.
    Choosing a time-independent test function $\phi\in C^{0,1}(\crc)$ on with $\|\partial_x\phi\|_{L^\infty}\le1$, we thus obtain, thanks to \eqref{eq:wL1}
  \begin{equation}
    \label{eq:BLbyK}
    \begin{split}
      \intom \phi\big(\rho_n(t_1)-\rho_n(t_0)\big)\dd x
      &= \int_{t_0}^{t_1}\intom\partial_x\phi\,\welo_n(t)\dd x\dd t \\
      & \le \int_{t_0}^{t_1}\|\partial_x\phi\|_{L^\infty}\intom \big|\welo_n(t)\big|\dd x\dd t \\
      & \le \int_{t_0}^{t_1}\big[2M\,\kinet\big(\rho_n(t),\welo_n(t)\big)\big]^{1/2}\dd t \\
      & \le (t_1-t_0)^{1/2}\left(2M\,\eps e^{T/\eps}\int_0^\infty\kinet\big(\rho_n(t),\welo_n(t)\big)\dd\mu^\eps(t)\right)^{1/2} \\
      &\le \big(2M\,\eps e^{T/\eps}\overline\nrg\big)^{1/2}(t_1-t_0)^{1/2},
    \end{split}
  \end{equation}
  and this is \eqref{eq:xBLhoelder}.
  % In particular,

  The locally uniform convergence \eqref{eq:xBLconv} to a continuous limit curve $\bar\rho^\eps:\Rnn\to(\prb,\BL)$
  is now obtained by means of the the Arzela-Ascoli theorem,
  see e.g.\ \cite[Proposition 3.3.1]{AGS} for an appropriate formulation in metric spaces.
  Indeed, \eqref{eq:xBLhoelder} implies equi-continuity of the curves $\rho_n$ on each interval $[0,T]$,
  and since $\crc$ is a compact set, the complete metric space $(\prb,\BL)$ is sequentially compact.
  From here, the convergence \eqref{eq:epsL1conv} in $L^1$ follows by interpolation
  of the locally uniform convergence \eqref{eq:xBLconv}
  with the $n$-uniform bound on $\rho_n$ in $L^2([0,T];H^1(\crc))$, i.e., from
  \begin{align*}
    \overline\nrg \ge \int_0^T\nrg(\rho)\dd\mu^\eps(t) \ge \frac{e^{-T/\eps}}{2\eps}\int_0^T\intom \big(\partial_x\rho\big)^2 \dd x \dd t.
  \end{align*}
  One can apply, for instance, a generalized version of the Aubin-Lions compactness theorem,
  see Theorem \ref{thm:savare} in the Appendix,
  using $\fnc:=\nrg$ and $g(\rho,\rho'):=\BL(\rho,\rho')$.
  This yields convergence in $L^1(\crc)$ in measure with respect to $t\in[0,T]$,  
  and since $\bar\rho^\eps$ is $t$-uniformly bounded in $L^1(\crc)$,
  also the strong convergence in $L^1((0,T)\times\crc)$.
  Moreover, since $\nrg$ is lower semi-continuous with respect to $\BL$-convergence,
  the limit in \eqref{eq:rhoconv2} follows by Fatou's lemma using that $\nrg$ is non-negative.
  The convergence with respect to $\BL$ implies in particular that
  \begin{align*}
    \bar\rho^\eps(0) = \lim_{n_k\to\infty}\rho_{n_k}(0) = \bar\rho_0.
  \end{align*}
  Thus, $\bar\rho^\eps$ satisfies properties (C1) and (C2).
  % 
  % For ease of notation, and without loss of generality,
  % we assume that the convergence \eqref{eq:xBLconv} is actually true for the entire sequence $(\rho_n)_{n\ge1}$.
  \medskip
     
  \emph{Construction of $\bar\welo^\eps$.}
  We shall now prove that there is a further subsequence $(\rho_{n_\ell},\welo_{n_\ell})_{\ell\ge1}$
  and a locally integrable limit $\bar\welo^\eps:\Rnn\to L^1(\crc)$
  such that
  \begin{align}
    \label{eq:weloconv1}
    &\welo_{n_\ell} \to \bar\welo^\eps \quad \text{as distribution on $\Rnn\times\crc$, and} \\
    \label{eq:weloconv2}
    &\int_0^\infty\kinet(\bar\rho^\eps(t),\bar\welo^\eps(t))\dd\mu^\eps(t)
    \le \liminf_{\ell\to\infty}\int_0^\infty\kinet(\rho_{n_\ell}(t),\welo_{n_\ell}(t))\dd\mu^\eps(t).
  \end{align}  
  Define an auxiliary density $\eta_n$ on $\Rnn\times\crc\to\Rnn$ by
  \begin{align*}
    \eta_n(t;x) = \frac{e^{-t/\eps}}{\eps}\mob(\rho_n(t;x)).
  \end{align*}
  From $L^1$-convergence \eqref{eq:epsL1conv} of $\rho_n$ in space and time,
  it follows that $\eta_n$ converges in each $L^1([0,T]\times\crc)$ to 
  % The convergence \eqref{eq:xBLconv} in combination with the tightness estimate
  % \begin{align*}
  %   \int_T^\infty \intom\dd\eta_n(t;x) = \int_T^\infty\intom \mob(\rho_n(t;x)) \frac{e^{-t/\eps}}{\eps}\dd x\dd t \le M\frac{e^{-T/\eps}}{\eps}
  %   \quad \text{for all $T>0$}
  % \end{align*}
  % with $M>0$ from \eqref{eq:wL1}
  % is sufficient to conclude that the narrow convergence of $\eta_n$
  % to the density $\eta^\eps$ given by
  \begin{align*}
    \eta^\eps(t;x) = \frac{e^{-t/\eps}}\eps \mob(\bar\rho^\eps(t;x)).
  \end{align*}  
  Further, introduce $\velo_n:\Rnn\times\crc\to\R$ by
  \begin{align*}
    \velo_n(t;x) =
    \begin{cases}
      \frac{\welo_n(t;x)}{\mob(\rho_n(t;x))} & \text{if $\mob(\rho_n(t;x))>0$}, \\
      0 & \text{if $\mob(\rho_n(t;x))=0$}.
    \end{cases}
  \end{align*}
  Since $\wed^\eps(\rho_n,\welo_n)\le\overline\nrg$,
  we have $\welo_n=0$ almost everywhere on $\{\mob(\rho_n)=0\}$,
  see Remark \ref{rmk:0by0},
  and thus $\welo_n = \mob(\rho_n)\velo_n$ a.e. on $\Rnn\times\crc$,
  and therefore also
  \begin{align*}
    \kfrac{\welo_n^2}{\mob(\rho_n)} = \velo_n^2\,\mob(\rho_n) \quad \text{a.e. on $\Rnn\times\crc$}.
  \end{align*}
  By definition of $\kinet$, we may thus conclude that
  \begin{align*}
    \overline\nrg
    \ge \int_0^\infty \kinet\big(\rho_n(t),\welo_n(t)\big)\dd\mu^\eps(t)
    = \frac12 \int_{\Rnn\times\crc}\velo_n^2(t;x)\dd\eta_n(t;x).
  \end{align*}
  The abstract convergence result \cite[Theorem 5.4.4]{AGS} for products of functions with measures
  provides the existence of a limiting function $\bar\velo^\eps\in L^2(\Rnn\times\crc;\eta^\eps)$
  and a subsequence $(\rho_{n_\ell},\welo_{n_\ell})_{\ell\ge1}$ such that
  \begin{align}
    \label{eq:weloconv1a}
    &\mob(\rho_{n_\ell})\,\velo_{n_\ell} \to \mob(\bar\rho^\eps)\,\bar\velo^\eps \quad\text{in the sense of distributions on  $\Rnn\times\crc$}, \\
    \label{eq:weloconv2a}
    &\int_{\Rnn\times\crc}\big(\bar\velo^\eps\big)^2\dd\eta^\eps \le \liminf_{\ell\to\infty}\int_{\Rnn\times\crc}\velo_{n_\ell}^2\dd\eta_{n_\ell}.
  \end{align}
  Define $\bar\welo^\eps:=\mob(\bar\rho^\eps)\,\bar\velo^\eps$, so in particular $\bar\welo^\eps=0$ on $\{\mob(\bar\rho^\eps)=0\}$.
  To verify (C3), i.e., that $\bar\welo^\eps\in \Lloc^1(\Rnn;L^1(\crc))$ for each $T>0$,
  simply observe that 
  \begin{align*}
    \int_0^T\intom |\bar\welo^\eps|\dd x\dd t
    &= \int_0^T\intom \mob(\bar\rho^\eps)|\bar\velo^\eps|\dd x\dd t \\
    &\le \eps e^{T/\eps}\int_{\Rnn\times\crc}|\bar\velo^\eps|\dd\eta^\eps 
    \le \eps e^{T/\eps} \left(M\int_{\Rnn\times\crc}\big(\bar\velo^\eps\big)^2\dd\eta^\eps\right)^{1/2}
      \le \eps e^{T/\eps}\sqrt{2M\,\overline\nrg}.
  \end{align*}
  by \eqref{eq:weloconv2a} just above.
  Now \eqref{eq:weloconv1} follows directly from \eqref{eq:weloconv1a} and the definition of $\velo$,
  and likewise, \eqref{eq:weloconv2} follows from \eqref{eq:weloconv2a},
  using that $\kfrac{(\bar\welo^\eps)^2}{\mob(\bar\rho^\eps)}=(\bar\velo^\eps)^2\,\mob(\bar\rho^\eps)$.
  \medskip
  
  \emph{Verification of (C4).}
  Each pair $(\rho_n,\welo_n)$ satisfies the distributional form \eqref{eq:weakcont} of the continuity equation.
  Since both $\rho_n$ and $\welo_n$ converge to their respective limits $\bar\rho^\eps$ and $\bar\welo^\eps$ as distributions,
  so we conclude \eqref{eq:weakcont} also for $(\mrho,\mwelo)$.
  \medskip
  
  \emph{Verification of \eqref{eq:xgoal}.}
  The limit property \eqref{eq:xgoal} is an immediate consequence
  of the estimates \eqref{eq:rhoconv2} and \eqref{eq:weloconv2},
  obtained in the construction of $\bar\rho^\eps$ and $\bar\welo^\eps$, respectively.
\end{proof}

%%%%%%%%%%%%%%%%%%%%%%%%%%%%%%%%%%%%%%%%%%%%%%%%%%%%%%%
\subsection{General a priori estimates}
%%%%%%%%%%%%%%%%%%%%%%%%%%%%%%%%%%%%%%%%%%%%%%%%%%%%%%%
%
The goal of this section is to prove the following a priori estimate that
holds in great generality for the WED approximations.
\begin{proposition}
  \label{prp:generalapriori}
  Let an initial datum $\bar\rho_0$ with $\nrg(\bar\rho_0)<\infty$,
  a finite time horizon $T>0$ be given.
  For any regularization $\eps>0$ and any global minimizer $(\bar\rho^\eps,\bar\welo^\eps)$ of $\wed^\eps$,
  the following bounds hold:
  \begin{align}
    \label{eq:trivialbound}
    \wed^\eps(\bar\rho^\eps,\bar\welo^\eps)&\le\nrg(\bar\rho_0), \\
    \label{eq:kinetbound}
    \int_0^T \kinet\big(\bar\rho^\eps_t,\bar\welo^\eps_t\big)\dd t &\le \frac12\nrg(\bar\rho_0), \\
    \label{eq:potenzbound}
    \int_0^T \big[\nrg\big(\bar\rho^\eps_t\big)-\nrg(\bar\rho_0)\big]\dd t &\le \frac\eps 2\nrg(\bar\rho_0).
  \end{align}
\end{proposition}
The corresponding estimates in the Hilbertian setting are classical, see e.g. \cite{MielkeOrtiz,MielkeStefanelli}.
A sophisticated version for the context of abstract metric gradient flows is given in \cite[Corollary 4.4]{RSSS}.
Since our setting is more general than the first approach, but does not require the full machinery of the second,
we provide a proof of Proposition \ref{prp:generalapriori} by standards methods,
following the strategy from \cite{Simon}.
\begin{proof}
  The bound \eqref{eq:trivialbound} follows immediately
  by comparison of the minimal value $\wed^\eps(\bar\rho^\eps,\bar\welo^\eps)$
  with the value generated by the admissible curve $(\rho^*,\welo^*)$
  with $\rho^*_t=\bar\rho_0$ and $\welo^*_t=0$ for all $t>0$.
  Since $\kinet(\rho^*,\welo^*)\equiv0$, it follows that
  $\wed^\eps(\rho^*,\welo^*)=\nrg(\bar\rho_0)$.

  The proof of \eqref{eq:kinetbound} is more involved.
  We start by showing that, in the sense of distributions,
  \begin{align}
    \label{eq:almostfundamental}
    \partial_t\big[\nrg(\bar\rho^\eps)-\eps\kinet(\bar\rho^\eps,\bar\welo^\eps)\big] = -2\kinet(\bar\rho^\eps,\bar\welo^\eps).
  \end{align}
  This is obtained by means of the so-called inner variation in time:
  let $\sigma\in C^\infty_c(\Rp)$,
  then the map $T_s:\Rnn\to\Rnn$ with $T_s(t)=t+s\sigma(t)$ is a diffeomorphism
  for all sufficiently small $s>0$.
  Define accordingly a perturbation $(\rho^s,\welo^s)$ of the minimizers $(\bar\rho^\eps,\bar\welo^\eps)$ by
  \begin{align*}
    \rho^s(t;x) = \bar\rho^\eps\big(T_s^{-1}(t);x\big),
    \quad
    \welo^s(t;x) = \frac{\bar\welo^\eps\big(T_s^{-1}(t);x\big)}{\dot T_s\big(T_s^{-1}(t)\big)}.
  \end{align*}
  The perturbed curves still lie in $\curves(\bar\rho_0)$;
  in particular, the continuity equation is satisfied since, in weak sense,
  \begin{align*}
    \partial_t\rho^s(t;x)
    = \frac{\partial_t\bar\rho^\eps\big(T_s^{-1}(t);x\big)}{\dot T_s\big(T_s^{-1}(t)\big)}
    = - \frac{\partial_x\bar\welo^\eps\big(T_s^{-1}(t);x\big)}{\dot T_s\big(T_s^{-1}(t)\big)}
    = - \partial_x\welo^s(t;x).
  \end{align*}
  Plugging into $\wed^\eps$, we obtain via a change of variables $t=T_s(t')$:
  \begin{align*}
    F(s)&:=\wed^\eps\big(\rho^s,\welo^s\big) \\
    &= \int_0^\infty \big[\eps\kinet\big(\rho^s(t),\welo^s(t)\big) + \nrg\big(\rho^s(t)\big)\big]\dd\mu^\eps(t)\\
    &= \int_0^\infty \left[\frac\eps{\dot T_s^2}\kinet(\bar\rho^\eps,\bar\welo^\eps)+\nrg(\bar\rho^\eps)\right]\circ T_s^{-1}(t) \frac{e^{-t/\eps}}{\eps}\dd t \\
    &= \int_0^\infty \left[\frac\eps{\dot T_s(t')}\kinet\big(\bar\rho^\eps(t'),\bar\welo^\eps(t')\big)+\dot T_s(t')\,\nrg\big(\bar\rho^\eps(t')\big)\right]\frac{e^{-T_s(t')/\eps}}{\eps}\dd t'.
  \end{align*}
  Observe that $\dot T_s(t')=1+s\dot\sigma(t')$ and $\partial_s\exp(-T_s(t')/\eps)=-\exp(-T_s(t')/\eps)\sigma(t')/\eps$.
  By dominated convergence, $F(s)$ is continuously differentiable at $s=0$, with $F(0)=\wed^\eps(\bar\rho^\eps,\bar\welo^\eps)$, and,
  since $(\bar\rho^\eps,\bar\welo^\eps)$ is a minimizer,
  \begin{align*}
    0 = F'(0)
    = \int_0^\infty\left(
    \dot\sigma\big[ -\eps\kinet(\bar\rho^\eps,\bar\welo^\eps)+\nrg(\bar\rho^\eps)\big]
    -\frac{\sigma}{\eps} \big[\eps\kinet(\bar\rho^\eps,\bar\welo^\eps)+\nrg(\bar\rho^\eps)\big]\right)
    \frac{e^{-t'/\eps}}{\eps}\dd t'.
  \end{align*}
  This formula is simplified by introducing $\psi\in C^\infty_c(\Rp)$ by $\sigma(t') = \eps e^{t'/\eps}\psi(t')$.
  Since
  \begin{align*}
    \dot\sigma(t') = e^{t'/\eps}\big(\eps\,\dot\psi(t')+\psi(t')\big),
  \end{align*}
  we obtain
  \begin{align*}
    \int_0^\infty (-\dot\psi)\big[ \eps\kinet(\bar\rho^\eps,\bar\welo^\eps)-\nrg(\bar\rho^\eps)\big]\dd t'
    = 2\int_0^\infty\psi\, \kinet(\bar\rho^\eps,\bar\welo^\eps)\dd t',
  \end{align*}
  which is the weak formulation of \eqref{eq:almostfundamental}.
  
  Since $\kinet(\bar\rho^\eps,\bar\welo^\eps)$ is integrable with respect to $\mu^\eps$,
  and thus locally integrable with respect to the Lebesgue measure on $\Rnn$,
  the derivative is actually in $\Lloc^1(\Rnn)$.
  Hence, $f:=\nrg(\bar\rho^\eps)-\eps\kinet(\bar\rho^\eps,\bar\welo^\eps)\in \Wloc^{1,1}(\Rnn)$
  and in particular, we may assume that $f$ is continuous,
  and is non-increasing by \eqref{eq:almostfundamental}.
  Moreover,
  \begin{align*}
    \partial_t\big(e^{-t/\eps}f(t)\big)
    = \frac{e^{-t/\eps}}\eps \big( \eps\partial_tf(t) - f(t)\big)
    = -\frac{e^{-t/\eps}}\eps\big[\eps\kinet(\bar\rho^\eps,\bar\welo^\eps)+\nrg(\bar\rho^\eps)\big] \le 0
  \end{align*}
  in the sense of weak derivatives, thus $t\mapsto e^{-t/\eps}f(t)$ is continuous and non-increasing as well.
  Since
  \begin{align*}
    \frac1\eps\int_0^\infty e^{-t/\eps}\big|f(t)\big|\dd t \le \wed^\eps(\bar\rho^\eps,\bar\welo^\eps),
  \end{align*}
  the monotone function $t\mapsto e^{-t/\eps}f(t)$ is in $L^1(\Rnn)$.
  Thus, it is non-negative, with
  \begin{align*}
    e^{-t/\eps}f(t) \searrow 0 \quad\text{as $t\to\infty$},
  \end{align*}
  so that we may conclude non-negativity of $f$ itself. It holds trivially $f(0) \le \nrg(\bar\rho_0)$, thus for each $T>0$
  \begin{align*}
    2\int_0^T\kinet(\bar\rho^\eps_t,\bar\welo^\eps_t)\dd t
    = f(0) - f(T)
    \le \nrg(\bar\rho_0),
  \end{align*}
  which is \eqref{eq:kinetbound}.
  The proof of \eqref{eq:potenzbound} now follows easily using the monotonicity of $f$:
  the inequality $f(t)\le f(0)$ for any $t\ge0$ implies
  \begin{align*}
    \nrg\big(\bar\rho^\eps(t)\big) - \nrg(\bar\rho_0) \le \eps\kinet\big(\bar\rho^\eps(t),\bar\welo^\eps(t)\big) ,
  \end{align*}
  so integration with respect to $t\in[0,T]$ and an application of \eqref{eq:kinetbound} leads to \eqref{eq:potenzbound}.
\end{proof}

% %%%%%%%%%%%%%%%%%%%%%%%%%%%%%%%%%%%%%%%%%%%%%%%%%%%%%%%
% \subsection{Application of the WED regularization}
% %%%%%%%%%%%%%%%%%%%%%%%%%%%%%%%%%%%%%%%%%%%%%%%%%%%%%%%
% %
% The $\eps$-elliptic regulatization of \eqref{eq:tf} is given by
% \begin{align}
%   \label{eq:WEDtf}
%   -\eps\left(\partial_{tt}\rho+\partial_{xx}\left[\frac{\welo^2}{\rho}\right]\right)+\partial_t\rho
%   = -\partial_x\big(\rho\,\partial_{xxx}\rho\big).
% \end{align}
% Notice that for sufficiently smooth $\rho$, the fourth order operator can be rewritten in the form
% \begin{align*}
%   \partial_x\big(\rho\,\partial_{xxx}\rho\big)
%   = \partial_{xxx}\big( \rho\,\partial_x\rho) - \partial_{xx}\left(\frac32(\partial_x\rho)^2\right).
% \end{align*}

%%%%%%%%%%%%%%%%%%%%%%%%%%%%%%%%%%%%%%%%%%%%%%%%%%%%%%%
%%%%%%%%%%%%%%%%%%%%%%%%%%%%%%%%%%%%%%%%%%%%%%%%%%%%%%%
\section{Thin film case --- linear mobility}
\label{sct:tf}
%%%%%%%%%%%%%%%%%%%%%%%%%%%%%%%%%%%%%%%%%%%%%%%%%%%%%%%
%%%%%%%%%%%%%%%%%%%%%%%%%%%%%%%%%%%%%%%%%%%%%%%%%%%%%%%
In this section, we study the WED approximation to \eqref{eq:intro-ch}
with linear mobility $\mob(\rho)=\rho$, i.e.,
\begin{equation}
  \label{eq:hs}
  \partial_t \rho_t + \partial_x(\rho_t\, \partial_{xxx}\rho_t) =0.
\end{equation}
By homogeneity, there is no loss of generality in assuming unit mass, $\mathrm m=1$,
i.e., $\rho_t$ is a probability density for each $t\ge0$.
Throughout this section and the next, we use a subscript $t$ to indicate evaluation at time $t$. 

%%%%%%%%%%%%%%%%%%%%%%%%%%%%%%%%%%%%%%%%%%%%%%%%%%%%%%%
\subsection{Euler-Lagrange Equation}
%%%%%%%%%%%%%%%%%%%%%%%%%%%%%%%%%%%%%%%%%%%%%%%%%%%%%%%
%
\begin{proposition}
  \label{prp:tfepsweak}
  For a given $\eps>0$, let $(\mrho,\mwelo)$ be a corresponding minimizer of $\wed^\eps$.
  Then the following weak form of the $\eps$-regularized evolution equation \eqref{eq:tf0weak} holds:
  \begin{equation}
    \label{eq:tfepsweak}
    0= \int_0^{\infty}\intom\left(
      \eps\left\{-\partial_{tt}\Phi_t\,\mrho_t
      + \partial_{xx}\Phi_t\,\frac{(\mwelo_t)^2}{\mrho_t}\right\}
      - \partial_t\Phi_t\,\mrho_t
      - \frac{3}{2}\partial_{xx} \Phi_t\,(\partial_x \mrho_t)^2
      -\partial_{xxx}\Phi_t\,\mrho_t\,\partial_x\mrho_t
    \right) \dd x \dd t
  \end{equation}
  for any test function $\Phi\in C^\infty_c(\Rp\times\crc)$.
\end{proposition}
The idea is to perform an inner variation of $(\mrho,\mwelo)$ in space, similarly as in \cite{JKO}.
To that end, let $\varphi\in C^\infty(\Rp\times\crc)$ be given;
in \eqref{eq:phiPhi} further below,
$\varphi$ will be related to the test function $\Phi$ in \eqref{eq:tfepsweak} above.
For any fixed $t>0$,
let $X_t^{(\cdot)}:\R\times\crc\to\crc$ be the flow map of $\partial_x\varphi_t$,
that is
\begin{align*}
  \partial_sX_t^s = \partial_x\varphi_t\circ X_t^s, \quad X_t^0 = \id.
\end{align*}
For later reference, we note that standard ODE arguments yield that
\begin{align}
  \label{eq:Xapprox}
  \partial_xX^s_t = 1+s\,\partial_{xx}\varphi_t + O(s^2),
  \quad
  \partial_{xx}X^s_t = s\,\partial_{xxx}\varphi_t + O(s^2),
  \quad
  \partial_tX^s_t = s\,\partial_t\partial_x\varphi_t + O(s^2),
\end{align}
where the implicit constants in $O(s^2)$ are uniformly bounded with respect to $s\in(0,1)$ and $t>0$,
thanks to smoothness and compact support of $\varphi$.
Next, define according pertubations $(\rho^s,\welo^s)$ of the minimizer $(\mrho,\mwelo)$ for $t>0$ by
\begin{align*}
  \rho^s_t:=X_t^s\#\mrho_t = \frac{\mrho_t}{\partial_xX_t^s}\circ(X_t^s)^{-1},\quad
  \welo^s_t:=\left(\mwelo_t+\partial_tX_t^s\frac{\mrho_t}{\partial_xX_t^s}\right)\circ(X_t^s)^{-1}. 
\end{align*}
\begin{lemma}
  For each $s\in\R$, we have $(\rho^s,\welo^s)\in\curves(\bar\rho_0)$.
\end{lemma}
\begin{proof}
  By smoothness and compact support of $\varphi$, it is easily seen that
  properties (C1)--(C3) are inherited from $(\mrho,\mwelo)$ to the perturbation $(\rho^s,\welo^s)$.
  It remains to verify (C4):
  for a test function $\theta\in C^\infty_c(\Rp\times\crc)$,
  \begin{align*}
    \int_0^\infty\intom \big(\partial_t\theta_t\,\rho^s_t+\partial_x\theta_t\,\welo^s_t\big)\dd x\dd t
    &= \int_0^\infty\intom \big(\partial_t\theta_t\,\rho^s_t+\partial_x\theta_t\,\welo^s_t\big)\circ X^s_t\,\partial_xX^s_t\dd x\dd t \\
    &= \int_0^\infty\intom \big(\partial_t\theta_t\circ X^s_t\,\mrho_t + \partial_x\theta_t\circ X^s_t\,(\partial_xX^s_t\,\mwelo_t + \partial_tX^s_t\,\mrho_t)\big)\dd x\dd t \\
    &= \int_0^\infty\intom \big( \partial_t(\theta_t\circ X^s_t)\,\mrho_t + \partial_x(\theta_t\circ X^s_t)\,\mwelo_t\big)\dd x\dd t
      = 0,
  \end{align*}
  using that $(\mrho,\mwelo)$ satisfies (C4),
  and since $\theta\circ X^s\in C^\infty_c(\Rp\times\crc)$.
\end{proof}
\begin{lemma}
  There is a remainder $\hat R^s_t$, which is uniformly bounded for $s\in(0,1)$ and $t>0$,
  such that at almost every $t>0$,
  \begin{align}
    \label{eq:tfvarykinet}
    \frac{\kinet(\rho^s_t,\welo^s_t)-\kinet(\mrho_t,\mwelo_t)}s
    = \intom \left(\partial_{xx}\varphi_t\,\frac{(\mwelo_t)^2}{\mrho_t}+\partial_t\partial_x\varphi_t\,\mwelo_t\right)\dd x
    + s\big[1+\kinet(\mrho_t,\mwelo_t)\big]\,\hat R^s_t
  \end{align}
  holds for all $s\in(0,1)$.
\end{lemma}
\begin{proof}
  Let $t>0$ be such that $\kinet(\mrho_t,\mwelo_t)<\infty$, i.e., $(\mwelo_t)^2/\mrho_t\in L^1(\crc)$.
  For any $s\in(-1,1)$, a change of variables $x=X_t^s(y)$ yields
  \begin{align*}
    \kinet(\rho^s_t,\welo^s_t)
    &= \intom \frac{(\welo^s_t)^2}{2\rho_t^s}\dd y \\
    &= \intom \frac{\big(\welo^s_t\circ X_t^s\big)^2}{2\rho_t^s\circ X_t^s}\partial_xX_t^s\dd y \\
    &= \intom \frac{\big(\partial_xX_t^s\,\mwelo_t+\partial_tX_t^s\mrho_t\big)^2}{2\mrho_t} \dd x \\
    &=\intom  \left\{(\partial_xX_t^s)^2\frac{(\mwelo_t)^2}{2\mrho_t}
      + \frac12(\partial_tX_t^s)^2\mrho_t
      + \partial_tX_t^s\,\partial_xX_t^s\,\mwelo_t\right\}\dd x.
  \end{align*}
  The expansion \eqref{eq:Xapprox} implies
  \begin{align*}
    (\partial_xX^s_t)^2 - 1 = 2s\,\partial_{xx}\varphi_t+O(s^2),
    \quad
    (\partial_tX^s_t)^2 = O(s^2),
    \quad
    \partial_tX^s_t\,\partial_xX_t^s = s\,\partial_t\partial_x\varphi_t + O(s^2),
  \end{align*}
  again with implicit constants that are uniformly bounded with respect to $s\in(0,1)$ and $t>0$.
  We thus obtain
  \begin{align*}
    &\left|\frac{\kinet(\rho^s_t,\welo^s_t)-\kinet(\mrho_t,\mwelo_t)}s
    - \intom \left(\partial_{xx}\varphi_t\,\frac{(\mwelo_t)^2}{\mrho_t}+\partial_t\partial_x\varphi_t\,\mwelo_t\right)\dd x \right| \\
    &\qquad = \left| \intom \left(O(s) \frac{(\mwelo_t)^2}{\mrho_t} + O(s)\mrho_t + O(s) \welo_t^s\right)\dd x\right|
    \le Cs\big[1+\kinet(\mrho_t,\mwelo_t)\big],
  \end{align*}
  with some constant $C$ that depends on $\varphi$ but is uniform in $s\in(0,1)$ and $t>0$.
  Above, we have used that $\mrho_t$ has unit mass and
  \begin{align*}
    \intom |\mwelo_t|\dd x
    \le \left(\intom \frac{(\mwelo_t)^2}{\mrho_t}\dd x\right)^{1/2}\left(\intom\mrho_t\dd x\right)^{1/2}
    \le \kinet(\mrho_t,\mwelo_t) + 1.
  \end{align*}
  Since $\kinet(\mrho_t,\mwelo_t)<\infty$ for almost every $t>0$, the claim follows.
\end{proof}
\begin{lemma}
  There is a remainder $\check R^s_t$, which is uniformly bounded for $s\in(0,1)$ and $t>0$,
  such that at almost every $t>0$,
  \begin{align}
    \label{eq:tfvarynrg}
    \frac{\nrg(\rho^s_t)-\nrg(\mrho_t)}s
    = -\intom \left(\frac32\partial_{xx}\varphi_t\,\big(\partial_x\mrho_t\big)^2+\partial_{xxx}\varphi_t\,\mrho_t\,\partial_x\mrho_t\right)\dd x
    + s\big[1+\nrg(\mrho_t)\big]\,\check R^s_t
  \end{align}
  holds for all $s\in(0,1)$.
\end{lemma}
\begin{proof}
  Let $t>0$ be such that $\nrg(\mrho_t)<\infty$, i.e., $\partial_x\mrho_t\in L^2(\crc)$.
  Note that, by interpolation,
  \begin{align}
    \label{eq:L2H1L1}
    \intom (\mrho_t)^2\dd x
    \le \frac12\intom\big(\partial_x\mrho_t\big)^2\dd x + B\intom \mrho_t\dd x
    = \nrg(\mrho_t) + B,
  \end{align}
  for some universal constant $B$, and in particular $\mrho_t\in H^1(\crc)$.  
  Now, by means of the chain rule for weak derivatives, one obtains
  \begin{align*}
    \partial_x\mrho_t
    = \partial_x\big(\partial_xX_t^s\,\rho^s_t\circ X_t^s\big)
    = (\partial_xX_t^s)^2(\partial_x\rho_t^s)\circ X_t^s + \partial_{xx}X_t^s\,\rho_t^s\circ X_t^s,
  \end{align*}
  and therefore
  \begin{align*}
    (\partial_x\rho_t^s)\circ X_t^s
    = \frac1{(\partial_xX_t^s)^{2}}\,\partial_x\mrho_t - \frac{\partial_{xx}X_t^s}{(\partial_xX_t^s)^3}\mrho_t.
  \end{align*}
  With a change of variables $x=X_t^s(y)$, it follows that
  \begin{align*}
    \nrg(\rho_t^s)
    &=\frac12\intom \big(\partial_x\rho^s_t\big)^2\dd y\\
    &= \frac12\intom \big((\partial_x\rho^s_t)\circ X_t^s\big)^2\partial_xX_t^s\dd y \\
    &= \frac12\intom \left(\frac1{(\partial_xX_t^s)^{2}}\,\partial_x\mrho_t - \frac{\partial_{xx}X_t^s}{(\partial_xX_t^s)^3}\mrho_t\right)^2\partial_xX_t^s\dd y \\
    &= \frac12\intom \left\{\frac1{(\partial_xX_t^s)^{3}}\,(\partial_x\mrho_t)^2
      + \frac{(\partial_{xx}X_t^s)^2}{(\partial_xX_t^s)^{5}}(\mrho_t)^2
      - 2\frac{\partial_{xx}X_t^s}{(\partial_xX_t^s)^{4}}\mrho_t\,(\partial_x\mrho_t)\right\}\dd x.
  \end{align*}
  The expansion \eqref{eq:Xapprox} implies that
  \begin{align*}
    \frac1{(\partial_xX^s_t)^{3}} - 1 = -3s\,\partial_{xx}\varphi_t + O(s^2),
    \quad
    \frac{(\partial_{xx}X^s_t)^2}{(\partial_xX^s_t)^{5}} = O(s^2),
    \quad
    - 2\frac{\partial_{xx}X^s_t}{(\partial_xX^s_t)^{4}} = -2s\partial_{xxx}\varphi_t + O(s^2),
  \end{align*}
  again with implicit constants that are uniformly bounded with respect to $s\in(-1,1)$ and $t>0$.
  We thus obtain
  \begin{align*}
    & \left|\frac{\nrg(\rho^s_t)-\nrg(\mrho_t)}s
    + \intom \left(\frac32\partial_{xx}\varphi_t\,\big(\partial_x\mrho_t\big)^2+\partial_{xxx}\varphi_t\,\mrho_t\,\partial_x\mrho_t\right)\dd x \right| \\
    &\qquad = \left|\intom\left(O(s) \,(\partial_x\mrho_t)^2 + O(s)\, (\mrho_t)^2 + O(s)\, \mrho_t\,\partial_x\mrho_t\right)\dd x\right| 
    \le Cs\big[1+\nrg(\mrho_t)\big]
  \end{align*}
  with some constant $C$ that depends on $\varphi$ but is uniform in $s\in(0,1)$ and $t>0$.
  Above, we have used the interpolation \eqref{eq:L2H1L1}, which implies in particular that
  \begin{align*}
    \intom \mrho_t\,\partial_x\mrho_t\dd x
    \le\left(\intom\big(\partial_x\mrho_t\big)^2\dd x\right)^{1/2}\left(\intom\big(\mrho_t\big)^2\dd x\right)^{1/2}
    \le  \nrg(\mrho_t)+\intom\big(\mrho_t\big)^2\dd x
    \le 2\nrg(\mrho_t) + B.
  \end{align*}
  Since $\nrg(\mrho_t)<\infty$ for almost every $t>0$, the claim follows.
\end{proof}
\begin{lemma}
  We have that
  \begin{align}
    \label{eq:tfpreweak}
    0 = \int_0^\infty\frac{e^{-t/\eps}}\eps \intom\left[
    \eps\left(\partial_{xx}\varphi_t\,\frac{(\mwelo_t)^2}{\mrho_t}+\partial_t\partial_x\varphi_t\,\mwelo_t\right)
    - \left(\frac32\partial_{xx}\varphi_t\,\big(\partial_x\mrho_t\big)^2+\partial_{xxx}\varphi_t\,\mrho_t\,\partial_x\mrho_t\right)
    \right]\dd x \dd t.
  \end{align}
\end{lemma}
\begin{proof}
  We combine the equations \eqref{eq:tfvarykinet} and \eqref{eq:tfvarynrg} as follows:
  by minimality of $(\mrho,\mwelo)$ for $\wed^\eps$, we have for every $s\in(0,1)$ that
  \begin{align*}
    0& \le \frac{\wed^\eps(\rho^s,\welo^s)-\wed(\mrho,\mwelo)}s \\
     &\le \int_0^\infty \left(\eps\frac{\kinet(\rho^s_t,\welo^s_t)-\kinet(\mrho_t,\mwelo_t)}s+\frac{\nrg(\rho^s_t)-\nrg(\mrho_t)}s\right)\dd\mu^\eps(t) \\
     &= \int_0^\infty\frac{e^{-t/\eps}}\eps \intom\left[
    \eps\left(\partial_{xx}\varphi_t\,\frac{(\mwelo_t)^2}{\mrho_t}+\partial_t\partial_x\varphi_t\,\mwelo_t\right)
    - \left(\frac32\partial_{xx}\varphi_t\,\big(\partial_x\mrho_t\big)^2+\partial_{xxx}\varphi_t\,\mrho_t\,\partial_x\mrho_t\right)
       \right]\dd x \dd t \\
    &\qquad + s\,\int_0^\infty\Big(\eps\hat R\big[1+\kinet(\mrho_t,\mwelo_t)\big]+\check R\big[1+\nrg(\mrho_t)\big]\Big)\dd\mu^\eps(t).
  \end{align*}
  The integral expression following $s$ in the last line is finite and independent of $s$,
  since $\hat R$, $\check R$ are uniformly bounded for $s\in(0,1)$ and $t>0$,
  since $\wed^\eps(\mrho,\mwelo)<\infty$, and since $\mu^\eps$ is a finite measure.
  Thus, we can simply pass to the limit $s\searrow0$,
  which yields \eqref{eq:tfpreweak} with ``$\le$'' in place of ``$=$'';
  equality is obtained by observing that the expression on the right-hand side is linear in $\varphi$,
  and that the inequality holds with $-\varphi$ in place of $\varphi$ as well.
\end{proof}
To finish the proof of Proposition \ref{prp:tfepsweak},
choose the test function $\varphi$ in the form
\begin{align}
  \label{eq:phiPhi}
  \varphi(t,x):=\eps e^{t/\eps}\Phi(t,x),
\end{align}
with $\Phi\in C^\infty_c(\Rp\times\crc)$ being the desired test function for \eqref{eq:tfepsweak}.
\begin{proof}[Proof of Proposition \ref{prp:tfepsweak}]
  Simply subsitute $\varphi=\eps e^{t/\eps}\Phi$ in \eqref{eq:tfpreweak}.
  The only term that needs further work is the one with $\partial_t\partial_x\varphi$.
  Using that
  \begin{align*}
    \partial_t\varphi = e^{t/\eps}(\Phi+\eps\,\partial_t\Phi),
  \end{align*}
  and the weak form \eqref{eq:weakcont} of the continuity equation $\partial_t\mrho+\partial_x\mwelo=0$,
  it follows that
  \begin{align*}
    \int_0^\infty e^{-t/\eps}\intom \partial_x\partial_t\varphi_t\,\mwelo_t \dd x\dd t
    &= \int_0^\infty \intom \partial_x\big(\Phi_t+\eps\,\partial_t\Phi_t\big)\,\mwelo_t\dd x\dd t\\&
    = -\int_0^\infty \intom\big(\partial_t\Phi_t+\eps\,\partial_{tt}\Phi_t\big)\,\mrho_t\dd x\dd t.
  \end{align*}
  This yields the desired weak form \eqref{eq:tfepsweak}.
\end{proof}

%%%%%%%%%%%%%%%%%%%%%%%%%%%%%%%%%%%%%%%%%%%%%%%%%%%%%%%
\subsection{Compactness estimate}
%%%%%%%%%%%%%%%%%%%%%%%%%%%%%%%%%%%%%%%%%%%%%%%%%%%%%%%
%
\begin{proposition}
  \label{prp:tfH2}
  Any minimizer $(\mrho,\mwelo)$ of $\wed^\eps$ satisfies $\mrho\in L^2_\loc((0,\infty);H^2(\crc))$,
  and there is an $\eps$-uniform constant $K$ such that
  for each $T>4$ and each $\tau\in(0,1)$,
  \begin{align}
    \label{eq:tfH2}
    \int_{2\tau}^{T-2\tau} \intom \big(\partial_{xx}\mrho_t\big)^2\dd x \dd t \le K\big(1+(\eps/\tau)^2\big)T.
  \end{align}
\end{proposition}
The prove uses a family of perturbations $(\rho^s,\welo^s)$ of the minimizer $(\mrho,\mwelo)$ by means of the heat flow.
Recall the definition of the periodic heat kernel $\htk^{(\cdot)}:\Rp\times\crc\to\crc$:
\begin{align*}
  \htk^\tau(z) = \tau^{-1/2}\htk\big(\tau^{-1/2}z\big),
  \quad\text{with}\quad
  \htk(\zeta) =(4\pi)^{-1/2}\sum_{m\in\Z} \exp\left(-\frac14(\zeta-m)^2\right).
\end{align*}
Let $\psi\in C^\infty_c(\Rp)$ be a test function and introduce for every $t>0$ the modified kernels $\mk_t^{(\cdot)}$ by
\begin{align}
  \label{eq:Gtilde}
  \mk_t^s = \htk^{s\psi(t)}.
\end{align}  
By the properties of the heat semi-group,
\begin{align}
  \label{eq:convol009}
  \partial_s\mk_t^s = \psi_t\,\partial_{zz}\mk_t^s, \quad
  \partial_t\mk_t^s = s\psi_t'\,\partial_{zz}\mk_t^s.
\end{align}
Now define the following perturbation $(\rho^s_t,\welo^s_t)$ of the minimizer $(\mrho,\mwelo)$
at any given $t\ge0$:
\begin{align*}
  \rho_t^s = \mk_t^s\ast\mrho_t, \quad
  \welo_t^s = \mk_t^s\ast\mwelo_t - s\psi'_t\,\partial_x\mk_t^s\ast\mrho_t.
\end{align*}
\begin{lemma}
  \label{lem:heatperturb1}
  $(\rho^s,\welo^s)\in\curves$ for each $s>0$.
\end{lemma}
\begin{proof}[Proof of Lemma \ref{lem:heatperturb1}]
  By the smoothing properties of the heat semi-group,
  we have that $\rho_t^s,\,\welo_t^s\in C^\infty(\crc)$ and moreover also $\rho_t^s>0$,
  for arbitrary $s>0$, at every $t>0$ where $\psi(t)>0$,
  It is readily seen that $(\rho^s,\welo^s)$ satisfies properties (C1)--(C3) because $(\mrho,\mwelo)$ does.
  To prove (C4), let $\theta\in C^\infty_c(\Rp\times\crc)$,
  and observe that by the elementary properties of convolution,
  and in particular since $z\mapsto\mk_t^s(z)$ is even while $z\mapsto\partial_x\mk_t^s(z)$ is odd,
  \begin{align*}
    & \int_0^\infty\intom \big(\partial_t\theta_t\,\rho^s_t+\partial_x\theta_t\,\welo^s_t\big)\dd x\dd t \\
    &\qquad = \int_0^\infty\intom \big((\mk^s_t\ast\partial_t\theta_t)\,\mrho_t + (\mk^s_t\ast\partial_x\theta_t)\,\mwelo_t
      - (s\psi'_t\,\partial_x\mk^s_t\ast\partial_x\theta_t)\,\mrho_t\big)\dd x\dd t \\
    &\qquad = \int_0^\infty\intom \big(\partial_t(\mk^s_t\ast\theta_t)\mrho_t + \partial_x(\mk^s_t\ast\theta_t)\,\mwelo_t\big)\dd x\dd t
      = 0,
  \end{align*}
  where we we used the relation \eqref{eq:convol009} and the fact that $(\mrho,\mwelo)$ satisfies the continuity equation.
\end{proof}
We are now going to estimate the difference quotient
\begin{align*}
  \frac1s\big(\wed^\eps(\rho^s,\welo^s)-\wed^\eps(\mrho,\mwelo)\big)
\end{align*}
in several steps.
\begin{lemma}
  \label{lem:tfperturb2}
  Let $t>0$ so that $\kinet(\mrho_t,\mwelo_t)<\infty$.
  Then, for every $s>0$,
  %\begin{align}
  %  \label{eq:tfperturb2}
  %  \frac{\kinet(\rho^s_t,\welo^s_t)-\kinet(\mrho_t,\mwelo_t)}s
  %  \le |\psi_t'|\kinet(\mrho_t,\mwelo_t)
  %  + 2|\psi_t'|\big(1+s|\psi_t'|\big)\intom \big(\partial_x\sqrt{\rho_t^s}\big)^2\dd x.
  %\end{align}
  \begin{align}
    \label{eq:tfperturb2}
    \frac{\kinet(\rho^s_t,\welo^s_t)-\kinet(\mrho_t,\mwelo_t)}s
    \le |\psi_t'|\kinet(\mrho_t,\mwelo_t)
    + \frac32|\psi_t'|\big(1+s|\psi_t'|\big)\left(\intom \big(\partial_{xx}\rho_t^s\big)^2\dd x\right)^{1/2}.
  \end{align}
\end{lemma}
\begin{proof}[Proof of Lemma \ref{lem:tfperturb2}]
  By means of the elementary inequality 
  \begin{align*}
    (a+rb)^2 &=a^2+r^2b^2+2rab\le a^2+r^2b^2+|r|(a^2+b^2)\\
    &= (1+|r|)a^2 + |r|(1+|r|)b^2 
  \end{align*}
  for arbitrary $a,b,r\in\R$,
  we conclude that
  \begin{align*}
    \big(\welo^s)^2
    = \big(\mk^s\ast\mwelo-s\psi'\,\partial_z\mk^s\ast\mrho\big)^2
    \le \big(1+s|\psi'|\big)\big|\mk^s\ast\mwelo\big|^2 + s|\psi'|\big(1+s|\psi'|\big)\big(\partial_z\mk^s\ast\mrho\big)^2.
  \end{align*} 
  Combining this with Jensen's inequality, using the convexity of the quotient,
  the pertubation of $\wed$'s kinetic part can be estimated as follows:
  \begin{align*}
    \intom \frac{(\welo^s_t)^2}{\rho^s_t}\dd x
    &\le \big(1+s|\psi_t'|\big)\intom\frac{\big(\mk^s_t\ast\mwelo_t\big)^2}{\mk^s_t\ast\mrho_t}\dd x
      + s|\psi_t'|\big(1+s|\psi_t'|\big)\intom\frac{\big(\partial_z\mk^s_t\ast\mrho_t\big)^2}{\mk^s_t\ast\mrho_t} \dd x \\
    &\le \big(1+s|\psi_t'|\big)\intom \mk^s_t\ast\left(\frac{(\mwelo_t)^2}{\mrho_t}\right)\dd x
      + s|\psi_t'|\big(1+s|\psi_t'|\big)\intom\frac{\big(\partial_x(\mk^s_t\ast\mrho_t)\big)^2}{\mk^s_t\ast\mrho_t} \dd x \\
     &  \le (1+s|\psi_t'|)\intom \frac{(\mwelo_t)^2}{\mrho_t}\dd x
      + 3s|\psi_t'|\big(1+s|\psi_t'|\big)\left(\intom \big(\partial_{xx}\rho_t^s\big)^2\dd x\right)^{1/2},
  \end{align*}
  where we have used \eqref{eq:villani} in the last step.
  Subtraction of $\kinet(\mrho_t,\mwelo_t)$ and division by $2s>0$ yield \eqref{eq:tfperturb2}.
\end{proof}
Concerning the potential term, we have:
\begin{lemma}
  \label{lem:et}
  Fix some $t>0$ and consider the function $e_t:\Rp\to\Rnn$ given by $e_t(s):=\nrg(\rho_t^s)$.
  Then:
  \begin{enumerate}
  \item $e_t$ is non-increasing,
  \item $e_t$ is convex,
  \item If $\mrho_t\in H^1(\crc)$, then $\lim_{s\searrow0}e_t(s)=\nrg(\mrho_t)$,
    and consequently
    \begin{align}
      \label{eq:nrgquot}
      \frac{\nrg(\mrho_t)-\nrg(\rho_t^s)}s \ge \psi_t(t)\intom \big(\partial_{xx}\rho_t^s\big)^2\dd x.
    \end{align}
  \end{enumerate}
\end{lemma}
\begin{proof}[Proof of Lemma \ref{lem:et}]
  If $\psi_t=0$, then $\rho^s_t\equiv\mrho_t$, and $\nrg(\rho^s_t)$ is independent of $s>0$;
  there is nothing to show.
  Assume $\psi_t>0$.
  Then $(s,x)\mapsto\rho^s_t(x)$ is smooth in $s>0$ and $x\in\crc$,
  and in particular, $e_t$ is an infinitely often differentiable map,
  with first and second derivative given by --- recalling \eqref{eq:convol009} ---
  \begin{align*}
    e_t'(s)
    &= \psi_t\intom \partial_x\rho_t^s\,\partial_{xx}\big(\partial_x\rho_t^s\big)\dd x
      = -\psi_t\intom \big(\partial_{xx}(\mk^s_t\ast\rho_t)\big)^2\dd x, \\
    e_t''(s)
    &= -2\psi_t^2\intom \partial_{xx}\rho_t^s\,\partial_{xx}\big(\partial_{xx}\rho_t^s\big)\dd x
      = 2\psi_t^2\intom \big(\partial_{xxx}(\mk^s_t\ast\rho_t)\big)^2\dd x,
  \end{align*}
  at any $s>0$.
  This immediately implies monotonicity and convexity of $e_t$.
  
  If additionally $\mrho_t\in H^1(\crc)$,
  then $\rho_t^s\to\mrho_t$ in $H^1(\crc)$ as $s\searrow0$ by the properties of convolution,
  and thus also $\lim_{s\searrow0}e_t(s)=\nrg(\mrho_t)$.
  Using $e_t$'s convexity and differentiability, the above tangent formula implies for any $s>0$ that
  \begin{align*}
    \nrg(\mrho_t) \ge e_t(s) - se_t'(s),
  \end{align*}
  and this is \eqref{eq:nrgquot}.
\end{proof}
\begin{proof}[Proof of Proposition \ref{prp:tfH2}]
  Now choose
  \begin{align}
    \label{eq:psifromPsi}
    \psi_t=\eps e^{t/\eps}\Psi_t^2,
  \end{align}
  with a cut-off function $\Psi\in C^\infty_c(\Rp)$
  such that $0\le\Psi\le1$ with $\Psi\equiv1$ on $(2\tau,T-2\tau)$ and $\Psi\equiv0$ outside of $(0,T)$,
  and such that $|\Psi'|\le1/\tau$.
  Then
   \begin{align*}
    \big|\psi_t'\big| = e^{t/\eps}\Psi_t\big|\Psi_t+2\eps\Psi'_t\big|
    \le e^{t/\eps}\Psi_t\big(1+2\eps/\tau\big)
    \le e^{T/\eps}\big(1+2\eps/\tau\big),
  \end{align*}
  where the last estimate follows from $\Psi_t=0$ for $t\ge T$.
  Using that $(\mrho,\mwelo)$ minimizes $\wed^\eps$,
  and combining \eqref{eq:nrgquot} with \eqref{eq:tfperturb2},
  recalling that $\kinet(\mrho_t,\mwelo_t)<\infty$ and $\nrg(\mrho_t)<\infty$ for a.e. $t>0$,
  we have accordingly --- assuming without loss of generality that $0<s<e^{-T/\eps}/(1+2\eps/\tau)$:
  %\begin{equation}
  %  \label{eq:var001}
  %  \begin{split}
  %    0 &\le \frac{\wed^\eps(\rho^s,\welo^s)-\wed^\eps(\mrho,\mwelo)}s \\
  %    &\le \int_0^\infty \left\{ \big(1+2\eps/\tau\big)\Psi_t \left[\kinet(\mrho_t,\mwelo_t) + 4\intom\big(\partial_x\sqrt{\rho^s_t}\big)^2\dd x\right]
  %      - \Psi_t^2\intom \big(\partial_{xx}\rho^s_t\big)^2\dd x \right\}\dd t.
  %  \end{split}
  %\end{equation}
    \begin{equation}
    \label{eq:var001}
    \begin{split}
      0 &\le \frac{\wed^\eps(\rho^s,\welo^s)-\wed^\eps(\mrho,\mwelo)}s \\
      &\le \int_0^\infty \left\{ \big(1+2\eps/\tau\big)\Psi_t \left[\kinet(\mrho_t,\mwelo_t) + 3\left(\intom \big(\partial_{xx}\rho_t^s\big)^2\dd x\right)^{1/2}\right]
        - \Psi_t^2\intom \big(\partial_{xx}\rho^s_t\big)^2\dd x \right\}\dd t.
    \end{split}
  \end{equation}
  Next, we estimate the middle term employing Young's inequality:
    \begin{align*}
    3\int_0^T  \big(1+2\eps/\tau\big)\Psi_t \left(\intom \big(\partial_{xx}\rho_t^s\big)^2\dd x\right)^{1/2} \dd t
    \le  \frac92\big(1+2\eps/\tau\big)^2T + \frac12\int_0^T \Psi_t^2 \intom \big(\partial_{xx}\rho^s_t\big)^2\dd x\dd t.
  \end{align*} 
  Substitute this into \eqref{eq:var001} and rearrange terms to obtain:
  \begin{align*}
    \frac12\int_0^T\Psi_t^2\intom \big(\partial_{xx}\rho^s_t\big)^2\dd x\dd t
    \le  \big(1+2\eps/\tau\big)\int_0^T\Psi_t \kinet(\mrho_t,\mwelo_t)\dd t + \frac92 \big(1+2\eps/\tau\big)^2T.
  \end{align*}
  Recalling that $\Psi\equiv1$ on $(2\tau,T-2\tau)$ and $\Psi\le1$,
  and that the time integral of $\kinet$ is bounded according to \eqref{eq:kinetbound},
  this yields
  \begin{align*}
    \int_{2\tau}^{T-2\tau}\intom \big(\partial_{xx}\rho^s_t\big)^2\dd x\dd t \le K\big(1+(\eps/\tau)^2\big)T,
  \end{align*}
  where $K$ does not depend on $\eps>0$. % provided that $0<s<e^{-T/\eps}$.
  At each $t>0$, we have that $\rho^s_t\to\mrho_t$ in $L^1(\crc)$ as $s\searrow0$,
  so by lower semi-continuity of the $H^2$-seminorm,
  \begin{align*}
    \int_{2\tau}^{T-2\tau}\intom \big(\partial_{xx}\mrho_t\big)^2\dd x\dd t 
    \le \liminf_{s\searrow0}\int_{2\tau}^{T-2\tau}\intom \big(\partial_{xx}\rho^s_t\big)^2\dd x\dd t
    \le K\big(1+(\eps/\tau)^2\big)T,
  \end{align*}
  which is \eqref{eq:tfH2}.
\end{proof}

%%%%%%%%%%%%%%%%%%%%%%%%%%%%%%%%%%%%%%%%%%%%%%%%%%%%%%%
\subsection{Regularity in space and time}
%%%%%%%%%%%%%%%%%%%%%%%%%%%%%%%%%%%%%%%%%%%%%%%%%%%%%%%
%
The goal of this section is to prove Theorem \ref{thm:tfpde} on the regularity of minimizers $\mrho$.
We start with an auxiliary result.
\begin{corollary}
  The weak formulation \eqref{eq:tfepsweak} can equivalently be written in the form
  \begin{align}
    \label{eq:tfepsweak1}    
    0= \int_0^{\infty}\intom\left(
    \eps\left\{-\partial_{tt}\Phi_t\,\mrho_t
    + \partial_{xx}\Phi_t\,\frac{(\mwelo_t)^2}{\mrho_t}\right\}
    - \partial_t\Phi_t\,\mrho_t
    + \partial_{xx} \Phi_t\,\Big(\mrho_t\partial_{xx} \mrho_t-\frac12(\partial_x \mrho_t)^2\Big)
    \right) \dd x \dd t.
  \end{align}
\end{corollary}
\begin{proof}
  With the regularity $\mrho\in L^2_\loc((0,T);H^2(\crc))$ from Proposition \ref{prp:tfH2} at hand,
  we can integrate by parts in \eqref{eq:tfepsweak} as follows,
  \begin{align*}
    -\intom \partial_{xxx}\Phi_t\, \mrhot\,\partial_{x}\mrhot \dd x
    = \intom \partial_{xx}\Phi_t\,\big( (\partial_x \mrhot)^2+\mrhot\,\partial_{xx}\mrhot \big)\dd x.
  \end{align*}
  which yields the form \eqref{eq:tfepsweak1} above.
\end{proof}
\begin{proof}[Proof of Theorem \ref{thm:tfpde}]
  Existence of a minimizer $(\mrho,\mwelo)$ and regularity $\mrho\in L^2_\loc((0,\infty);H^2(\crc))$
  have already been shown, respectively, in Proposition \ref{prp:x1} and Proposition \ref{prp:tfH2}.
  It remains to verify that
  \begin{align}
    \label{eq:thatsnew}
    \mrho\in W^{2,1}_\loc\big((0,\infty); W^{2,\infty}(\crc)^\ast\big).
  \end{align}
  By (the weak form of) the continuity equation,
  \begin{align*}
    -\int_0^T\intom \partial_t\Phi_t\,\mrhot\dd x\dd t = \int_0^T\intom \partial_x\Phi_t\,\mwelot \dd x\dd t.
  \end{align*}
  Substitute this in \eqref{eq:tfepsweak1},
  and choose $\Phi(t;x)=\phi(x)\psi(t)$ with $\phi\in C^\infty(\crc)$ and $\psi\in C^\infty_c((0,T))$
  to obtain
  \begin{align*}
    \eps\int_0^T\psi''_t \left(\intom \phi_t\,\mrhot\dd x\right)\dd t
    = \int_0^T\psi_t \left(\intom\left[
    \phi_t'\,\mwelot
    +\phi''_t\,\left(\eps\frac{(\mwelot)^2}{\mrhot}+\mrhot\,\partial_{xx}\mrhot-\frac12(\partial_x\mrhot)^2\right)
    \right] \dd x\right)\dd t.
  \end{align*}
  Now use that $\mwelo$, $(\mwelo)^2/\mrho$, $(\partial_x\mrho)^2$ and $\mrho\,\partial_{xx}\mrho$ all belong to $L^1_\loc((0,\infty)\times\crc)$.
  This shows \eqref{eq:thatsnew}.
\end{proof}

%%%%%%%%%%%%%%%%%%%%%%%%%%%%%%%%%%%%%%%%%%%%%%%%%%%%%%% 
\subsection{Limit of vanishing regularization}
%%%%%%%%%%%%%%%%%%%%%%%%%%%%%%%%%%%%%%%%%%%%%%%%%%%%%%%
%
%
In this section, we prove Theorem \ref{thm:tflimit}.
The proof reduces essentially to an application of a generalized Aubin-Lions compactness theorem,
recalled in Theorem \ref{thm:savare} in the Appendix.
\begin{proof}[Proof of Theorem \ref{thm:tflimit}]
  Consider a vanishing sequence $(\eps_n)_{n\in\N}$ of regularization parameters.
  Fix some $T>0$ and $\tau\in(0,T/4)$ throughout the proof.
  We are going to show existence of a (non-relabeled) subsequence $(\eps_n)_{n\in\N}$
  such that $\bar\rho^{\eps_n}$ converges to a weakly continuous limit curve $\bar\rho^*:[0,T]\to\prbr{\Rnn}$
  \begin{itemize}
  \item strongly in $C([0,T];\BL)$,
  \item weakly in $L^2([2\tau,T-2\tau];H^2(\crc)\big)$, and
  \item strongly in $L^2([2\tau,T-2\tau];H^1(\crc)\big)$.
  \end{itemize}
  The claimed local convergence \eqref{eq:tflimit9} on the semi-infinite interval $(0,+\infty)$
  then follows by means of a diagonal argument with $T\to\infty$ and $\tau\searrow0$.
  
  To begin with, observe that
  from the $\eps$-uniform bound \eqref{eq:kinetbound} on the time-integrated kinetic energy of minimizers,
  one obtains along the same lines as in \eqref{eq:BLbyK} the following equi-continuity estimate:
  \begin{align}
    \label{eq:BLlimit}
    \BL\big(\mrho_{t_1},\mrho_{t_0}\big) \le \sqrt{\nrg(\bar\rho_0)}\, |t_1-t_0|^{1/2},
  \end{align}
  Further, recall that the metric space $(\prbr{\Rnn},\BL)$ is compact.
  The uniform convergence of a subsequence in $\BL$, uniformly with respect to time,
  to some $\bar\rho^*\in C([0,T];\BL)$ now follows
  by a metric version of the Arzela-Ascoli theorem, see e.g.~\cite[Proposition 3.3.1]{AGS}.
  Note that, in particular,
  \[ \lrho(0) = \lim_{\eps\searrow0}\mrho(0) = \bar\rho_0. \]  
  Next, without loss of generality, we may assume that $\eps_n<\tau$,
  so that the key estimate \eqref{eq:tfH2} simplifies to
  \begin{align}
    \label{eq:tfH2b}
    \int_{2\tau}^{T-2\tau} \intom \big(\partial_{xx}\mrho_t\big)^2\dd x \dd t \le 2KT.
  \end{align}
  By means of Alaoglu's theorem
  --- passing to a further subsequence if necessary ---
  $\bar\rho^{\eps_n}$ converges weakly in $L^2([2\tau,T-2\tau];H^2(\crc))$;
  that weak limit must coincide with $\bar\rho^*$ obtained above.

  To obtain also strong convergence in $L^2([2\tau,T-2\tau];H^1(\crc))$,
  the generalized Aubin-Lions theorem \ref{thm:savare} is applied with the following parameters:
  Banach space $X:=H^1(\crc)$, closed convex subset $U:=X\cap\prbr{I}$,
  as well as $\fnc$ and $g$ given, respectively, by 
  \begin{align*}
    \fnc(\rho) =
    \begin{cases}
      \intom \big(\partial_{xx}\rho\big)^2\dd x & \text{if $\rho\in H^2(\crc)$}, \\
      +\infty & \text{otherwise},
    \end{cases},
                \qquad
                g(\rho,\eta) = \BL(\rho,\eta).
  \end{align*}
  Then tightness \eqref{eq:tight} with respect to $\fnc$ follows immediately from estimate \eqref{eq:tfH2b} above,
  and integral equi-continuity \eqref{eq:equicont} with respect to $g$
  is a consequence of the H\"older estimate \eqref{eq:BLlimit}.
  A --- potentially further --- subsequence $\bar\rho^{\eps_n}$ converges in $H^1(\crc)$, in measure with respect to $t\in[2\tau,T-2\tau]$.
  Strong convergence in $L^2([2\tau,T-2\tau];H^1(\crc))$ is now easily concluded by means of Vitali's theorem,
  using that the interpolation inequality
  \begin{align*}
    \big\|\partial_x\mrho_t\big\|_{L^3(\crc)} \le A \big\|\partial_{xx}\mrho_t\big\|_{L^2(\crc)}^{2/3}\big\|\mrho_t\big\|_{L^1(\crc)}^{1/3},
  \end{align*}
  with some universal constant $A$
  in combination with the a priori estimate \eqref{eq:tfH2}
  implies that
  \begin{align*}
    \int_{2\tau}^{T-2\tau}\big\|\partial_x\mrhot\big\|_{L^2}^3\dd t
    \le \int_{2\tau}^{T-2\tau}\intom \big|\partial_x\mrhot\big|^3\dd x\dd t
    \le A^3\int_{2\tau}^{T-2\tau}\intom \big(\partial_{xx}\mrho_t\big)^2 \dd x\dd t
    \le A^3KT.
  \end{align*}
  As mentioned above, a diagonal argument for $T\to+\infty$ and $\tau\searrow0$
  produces a globally defined weakly continuous curve $\bar\rho^*:[0,\infty)\to\prbr{\Rnn}$
  and a vanishing sequence $(\eps^*_n)_{n\in\N}$ such that $\bar\rho^{\eps_n^*}$ converges to $\bar\rho^*$
  as stated in \eqref{eq:tflimit9}, i.e.,
  \begin{align}
    \label{eq:wrapup}
    \bar\rho^{\eps^*_n}\to\bar\rho^* \quad
    \text{weakly in $L^2_\loc((0,\infty);H^2(\crc))$}, \quad
    \text{strongly in $L^2_\loc((0,\infty);H^1(\crc))$}.
  \end{align} 
  Recall that estimate \eqref{eq:tfH2b} holds for all $\eps\le\tau$, with a uniform constant $K$;
  by lower semi-continuity of the $H^2(\crc)$-norm,
  it follows that the limit $\bar\rho^*$ is regular up to the origin:
  \begin{align*}
    \bar\rho^*\in L^2\big([0,\infty);H^2(\crc)\big).
  \end{align*}
  It remains to verify the weak formulation \eqref{eq:tf0weak}.
  To that end, fix some test function $\Phi\in C_c^\infty((0,\infty)\times\crc)$,
  and pass to the limit $\eps^*_n\to0$ in the $\eps$-regularized weak form \eqref{eq:tfepsweak}. 
  The terms with pre-factor $\eps$ in \eqref{eq:tfepsweak} vanish
  since both $\mrho$ and $(\mwelo)^2/\mrho$ are $\eps$-uniformly bounded in $L^1([0,T]\times\crc)$,
  see \eqref{eq:kinetbound}.
  For the remaining terms, the convergences in \eqref{eq:wrapup} are sufficient to pass to the respective limits.
\end{proof}

%%%%%%%%%%%%%%%%%%%%%%%%%%%%%%%%%%%%%%%%%%%%%%%%%%%%%%%
%%%%%%%%%%%%%%%%%%%%%%%%%%%%%%%%%%%%%%%%%%%%%%%%%%%%%%%
\section{Cahn-Hilliard case --- nonlinear mobility}
\label{sct:ch}
%%%%%%%%%%%%%%%%%%%%%%%%%%%%%%%%%%%%%%%%%%%%%%%%%%%%%%%
%%%%%%%%%%%%%%%%%%%%%%%%%%%%%%%%%%%%%%%%%%%%%%%%%%%%%%%
%
This section deals with the WED approximation to \eqref{eq:intro-ch}
with nonlinear mobility $\mob(\rho)$ which is a uniformly concave function on $I=[0,1]$
with $\mob(0)=\mob(1)=0$.
We work under the assumptions that
\begin{align}
  \label{eq:assumption_m}
  \mob(r)\le\frac12,\quad |\mob'(r)|\le1,\quad\mob''(r)\le-1\qquad\text{for all $r \in I$}.
\end{align}
A possible choice is clearly $\mob(r)=r(1-r)$, 
but in the proofs in this section would not simplify significantly by assuming $\mob$ of this special form.

%%%%%%%%%%%%%%%%%%%%%%%%%%%%%%%%%%%%%%%%%%%%%%%%%%%%%%%
\subsection{A regularization}
%%%%%%%%%%%%%%%%%%%%%%%%%%%%%%%%%%%%%%%%%%%%%%%%%%%%%%%
%
For technical reasons, we perform a regularization of the mobility:
we introduce a parameter $\mu=\mu(\eps)$
that satisfies
\begin{align}
  \label{eq:epsmu}
  0<\mu(\eps)\le\frac12 \quad \text{for each $\eps\in(0,1)$},\qquad
  \frac\eps{\mu(\eps)^3}\searrow0 \quad \text{as $\eps\searrow0$},
\end{align}
but is otherwise arbitrary.
Note that that $\mob+\mu\le1$.

Accordingly, we replace the WED functional $\wed^\eps$
by $\hat\wed^\eps$ with $\mob+\mu(\eps)$ in place of $\mob$,
i.e., the kinetic term $\kinet$ in $\hat\wed^\eps$ is
\begin{align}
  \label{eq:intromu}
  \hat\kinet(\rho,\welo) = \frac12\intom \frac{\welo^2}{\mob(\rho)+\mu(\eps)}\dd x,
\end{align}
where the fraction is now genuine since $\mob(r)+\mu(\eps)$ is positive for every $r\in I$.
The existence of a minimizer $(\mrho,\mwelo)\in\curves(\bar\rho_0)$
for this modified functional is still guaranteed by Proposition \ref{prp:x1},
and it satisfies in particular the constraint $\mrho(t;x)\in I$ for a.e. $(t;x)\in\Rnn\times\crc$. 

The modification \eqref{eq:intromu} of the kinetic term resembles the strictly parabolic regularization,
that has been performed e.g.~in \cite{EG}.
However, the analytic effect is very different, since the density constraint to $I=[0,1]$ is still active.
A priori, there is no reason to expect improved regularity of $\mrho$ near points
where the extremal values $0$ or $1$ are attained.
When we derive a priori estimates and Euler-Lagrange equations below,
we need to use very particular variations that respect the constraint.

%%%%%%%%%%%%%%%%%%%%%%%%%%%%%%%%%%%%%%%%%%%%%%%%%%%%%%%
\subsection{A priori estimates}
%%%%%%%%%%%%%%%%%%%%%%%%%%%%%%%%%%%%%%%%%%%%%%%%%%%%%%%
%  
The aim of this section is to establish the following regularity of $(\mrho,\mwelo)$. 
\begin{proposition}
  \label{prp:regularity}
  For each $\eps\in(0,1)$ and $T>4$, the corresponding minimizer $(\mrho,\mwelo)$ satisfies
  \[ \mrho \in L^2\big([2\eps,T-2\eps],H^2(\crc)\big), \quad \mwelo \in L^2\big([2\eps,T-2\eps],H^1(\crc)\big), \]
  and 
  \begin{align}
    \label{eq:regularity}
    \int_{2\eps}^{T-2\eps}\intom\left\{
    \frac12\big(\partial_{xx}\mrho\big)^2
    +\frac\eps4\intom\big(\mwelo\,\partial_x\mrho\big)^2
    + \frac{\eps\mu}8\intom(\partial_x\mwelo)^2
    \right\}\dd x\dd t
    \le 18T + 5\nrg(\bar\rho_0).
  \end{align}
\end{proposition}
Before proceeding to the proof of Proposition \ref{prp:regularity},
we note an immediate consequence of the appearance of the $L^2$-norm of $\partial_x\mwelo$ in \eqref{eq:regularity}:
\begin{corollary}
  \label{cor:wxL2}
  For each $\eps>0$,
  the continuity equation $\partial_t\mrho=-\partial_x\mwelo$ is an equality
  between quantities in $L^2((2\eps,T-2\eps)\times\crc)$.
  Consequently, $\mrho\in H^1((2\eps,T-2\eps)\times\crc)$.
\end{corollary}
In the following, let some terminal time $T>4$ be fixed.
Similar to the proof of Proposition \ref{prp:tfH2},
we consider perturbations $(\rho^s,\welo^s)$ of the minimizer $(\mrho,\mwelo)$ by means of a properly scaled heat flow.
Define $\psi\in C^\infty_c(\Rp)$ as in \eqref{eq:psifromPsi} for $\tau=\eps$, i.e.,
\begin{align}
  \label{eq:psifromPsi2}
  \psi_t=\eps e^{t/\eps}\Psi_t^2,
\end{align}
where $\Psi\in C^\infty_c(\Rp)$ is a cut-off function
with $\Psi\equiv1$ on $(2\eps,T-2\eps)$, with $\Psi\equiv0$ outside of $(0,T)$,
and such that $|\Psi'|\le\eps^{-1}$.
% For later reference, we state that, similarly as in \eqref{eq:est_psi_prime},
% \begin{align}
%   \label{eq:est_psi_prime2}
%   \big|\psi'_t|\le \frac3\eps\psi_t \le 3e^{t/\eps}.
% \end{align}
With the rescaled heat kernels $\mk$ from \eqref{eq:Gtilde},
the perturbation of $(\mrho,\mwelo)$ for parameter $s\in(0,1)$ at any given time $t>0$
is defined as
\begin{align*}
  \rho_t^s = \mk_t^s\ast\mrho_t, \quad
  \welo_t^s = \mk_t^s\ast\mwelo_t - s\psi'(t)\,\partial_x\mk_t^s\ast\mrho_t.
\end{align*}
Recalling the properties \eqref{eq:convol009} of $\mk$,
it follows that $\rho_t^s$ and $\welo_t^s$ are differentiable in $s\in(0,1)$,
with
\begin{align*}
  \partial_s \rho_t^s=\psi_t\,\partial_{xx}\rho_t^s, \quad
  \partial_s \welo_t^s=\psi_t\, \partial_{xx}\welo_t^s -\psi'_t\, \partial_{x}\rho_t^s.
\end{align*}
Define accordingly
\begin{align*}
  \hat\kinet_t^s
  :=\hat\kinet(\rho_t^s,\welo_t^s)
  =\frac{1}{2} \intom \frac{(\welo_t^s)^2}{\mob(\rho_t^s)+\mu}\dd x,
  \quad
  e_t^s := \nrg(\rho_t^s) =\frac{1}{2}\intom \big(\partial_{xx}\rho_t^s\big)^2\dd x.
\end{align*}
The proof of Proposition \ref{prp:regularity} will be obtained in a series of lemmas.
\begin{lemma}
  $(\rho^s,\welo^s)\in\curves$ for each $s>0$.
\end{lemma}
The proof is identical to the one for Lemma \ref{lem:heatperturb1}.
\begin{lemma}
  \label{lem:continuity}
  $\hat\kinet_t^s$ and $e_t^s$ are continuous at $s=0$.
  Moreover, $\hat\kinet_t^s\to\hat\kinet_t^0$ strongly in $L^1((0,T))$.
\end{lemma}
\begin{proof}
  Since $\mrho_t \in H^1(\crc)$ and $\mwelo_t\in L^2(\crc)$,
  we have $\rho_t^s \to \mrho_t$ in $H^1(\crc)$ and $\welo_t^s \to \mwelot$ in $L^2(\crc)$.
  In particular, $\rho_t^s \to \mrho_t$ in measure on $\crc$ and $(\welo_t^s)^2 \to (\mwelot)^2$ strongly in $L^1(\crc)$,
  so that
  \begin{align*}
    \intom \frac{(\welo_t^s)^2}{\mob(\rho_t^s)+\mu}\dd x \to  \intom \frac{(\mwelo_t)^2}{\mob(\mrho_t)+\mu}\dd x.
  \end{align*}
  We have used that $(\rho,\welo)\mapsto\frac{\welo^2}{\mob(\rho)+\mu}$ is a bounded continuous map for $\mu>0$.

  Since $\mwelo$ is square integrable on $(0,T)\times\crc$, one even has $\welo^s\to\mwelo$ in $L^2((0,T)\times\crc)$.
  And since naturally $\rho^s\to\mrho$ in measure on $(0,T)\times\crc$,
  the quotient $(\welo^s)^2/(\mob(\rho^s)+\mu)$ converges in $L^1((0,T)\times\crc)$,
  which implies $L^1$-convergence of $\hat\kinet^s$.
\end{proof}
\begin{lemma}
  \label{lem:diff}
  $\hat\kinet_t^s$ and $e_t^s$ are continuously differentiable for $s>0$ with
  \begin{equation}
    \label{eq:dds}
    \begin{split}
      -\frac{e^{-t/\eps}}{\eps}
      &\,\frac{\dd}{\dn s}\big(\eps\hat\kinet^s_t+e^s_t)
      \ge - 9 \hat\kinet_t^s - 18  \\
      &+\Psi(t)^2
      \intom\left\{
        \frac12\big(\partial_{xx}\rhost\big)^2
        +\frac\eps4\intom\big(\welost\big)^2(\partial_x\rhost)^2\dd x
        + \frac{\eps\mu}8\intom\big(\partial_x\welost\big)^2\dd x
      \right\}\dd x
    \end{split}
  \end{equation}
  at every $t\in(0,T)$.
  % 
  % \begin{align}
  %   \label{eq:dds1} % \label{eq:estimate_partial_s}
  %   -\frac{\dd}{\dd s} e_t^s
  %   =\underline e_t^s
  %   &:=\psi_t\intom \big(\partial_{xx}\rho_t^s\big)^2\dd x,\\
  %   \label{eq:dds2}
  %   -\frac{\dd}{\dd s} \hat\kinet_t^s\ge \underline{\hat \kinet}_t^s
  %   &:=\frac\mu8\psi_t \intom\big(\partial_x \welo_t^s\big)^2 \dd x
  %     -\frac{\psi_t}{2\eps}\intom \big(\partial_{xx}\rho_t^s\big)^2\dd x
  %     -3e^{t/\eps} \hat\kinet_t^0 - 216 \, e^{t/\eps}\\
  %   \label{eq:dds3}
  %   & +\psi_t \intom \frac{-\mobstpp}4\left(\frac{\welo_t^s}{\mob(\rho_t^s)+\mu}\right)^2 \big(\partial_{x}\rho_t^s\big)^2 \dd x.
  % \end{align}
\end{lemma}
\begin{proof}
  Again, the smoothing effect of the heat flow justifies the manipulations below.
  The variation of the energy is easily computed:
  \begin{align*}
    I_0:=-\frac{e^{-t/\eps}}\eps\frac{\dd}{\dd s} e_t^s
    =-\frac{e^{-t/\eps}}\eps\intom \partialx\rhost\partialx(\psi(t)\partialxx\rhost)\dd x
    =\Psi(t)^2\intom \big(\partial_{xx}\rho_t^s\big)^2\dd x.
  \end{align*}
  The variation of the kinetic term produces:
  \begin{align*}
    -\frac{\dd}{\dd s} \hat\kinet_t^s
    &=-\intom\left[\kinint\partial_s \welost-\frac{\mob'(\rhost)}{2}\left( \kinint\right)^2\partial_s\rhost \right]\dd x\\
    &=-\intom\left[\kinint (\psit\partialxx\welost-\psit'\partialx\rhost)-\frac{\mob'(\rhost)}{2}\left( \kinint\right)^2\psit\partialxx\rho\right]\dd x.
  \end{align*}
  Integrate by parts in $x$ to reduce the double derivatives on $\rho$ and on $\welo$, respectively, to single derivatives,
  and multiply by $e^{-t/\eps}/\eps$.
  The resulting integral is the sum of the following three terms:
  \begin{align*}
    I_1&:=\eps\Psi(t)^2\intom\partialx\left(\kinint\right)\left\{ \partialx\welost-\kinint \mob'(\rhost)\partialx\rhost\right\}\dd x, \\
    I_2&:=\eps\Psi(t)^2\intom \frac{-\mobstpp}2 \left(\kinint\right)^2 \big(\partial_{x}\rho_t^s\big)^2\dd x, \\
    I_3&:=\Psi(t)\big(\Psi(t)+2\eps\Psi'(t)\big)\intom\kinint\partialx\rhost\dd x.
  \end{align*}
  Concerning $I_1$, observe that the expression in the curly bracket is the same, up to a factor $\mobst+\mu$,
  as the $x$-derivative of $\welost/(\mob(\rhost)+\mu)$.
  We thus obtain a square under the integral, which we can bound from below by the binomial formula:
  \begin{align*}
    I_1
    &= \eps\Psi(t)^2\intom(\mobst+\mu)\,\left(\frac{\partialx\welost}{\mobst+\mu}
      -\kinint\frac{\mobstp}{\mobst+\mu}\partialx\rhost\right)^2\dd x  \\
    &\ge\Psi(t)^2\left\{\frac{\eps}2\intom\frac{\big(\partialx\welost\big)^2}{\mobst+\mu}\dd x
      -\eps\intom\frac{\big(\mobstp\big)^2}{\mobst+\mu}\left(\frac{\welost}{\mob(\rhost)+\mu}\right)^2\big(\partialx\rhost\big)^2 \dd x\right\}.
  \end{align*}
  As an integral over a square, $I_1$ is non-negative, and so $I_1\ge\frac\mu4 I_1$ since $\mu<1$.
  Recalling $|\mob'|\le 1$, and using $-\mob''\ge1$ to estimate the expression for $I_2$ from below,
  it follows that
  \begin{align}
    \label{eq:I1I2}
    \begin{split}
      I_1+I_2 \ge \frac\mu4 I_1 + I_2
      &\ge \Psi(t)^2\left\{\frac{\mu\eps}8\intom\frac{\big(\partialx\welost\big)^2}{\mobst+\mu}\dd x
        + \frac{\eps}{4}\intom \left(\kinint\right)^2 \big(\partial_{x}\rho_t^s\big)^2\dd x\right\} \\
      &\ge \Psi(t)^2\left\{\frac{\mu\eps}8\intom \big(\partialx\welost\big)^2\dd x
        + \frac{\eps}{4}\intom \big(\welost\big)^2 \big(\partial_{x}\rhost\big)^2\dd x\right\},
    \end{split}
  \end{align}
  where the last inequality is a consequence of $\mob+\mu\le 1$.

  Concerning $I_3$, we start by noting that since $\mob(0)=\mob(1)=0$ and $-\mob''\ge1$,
  \begin{align*}
    \mob(r) \ge \frac{r(1-r)}2,
    \quad\text{and so}\quad
    -\frac1{\mob(r)+\mu} \ge -\frac2r - \frac2{1-r}.
  \end{align*}
  And since $\Psi\le1$ and $\eps|\Psi'|\le1$, we obtain for every $t\in(0,T)$:
  \begin{align}
    \begin{split}
  \label{eq:I3}
    I_3
    &\ge -\frac{\Psi(t)^2}2\intom \frac{\big(\partial_x\rhost\big)^2}{\mob(\rhost)+\mu}\dd x
      - \frac{\big(\Psi(t)+2\eps\Psi'(t)\big)^2}2\intom \frac{\big(\welost\big)^2}{\mob(\rhost)+\mu}\dd x \\
    &\ge -\left(\intom \frac{\big(\partial_x\rhost\big)^2}{\rhost}\dd x + \intom \frac{\big(\partial_x\rhost\big)^2}{1-\rhost}\dd x\right)
      - 9 \hat\kinet_t^s.
      \end{split}
  \end{align}
  We estimate further, using \eqref{eq:villani} with $f:=\rhost$ and $f:=1-\rhost:$
    \begin{align*}
    \intom \frac{\big(\partial_x\rhost\big)^2}{\rhost}\dd x + \intom \frac{\big(\partial_x\rhost\big)^2}{1-\rhost}\dd x 
    &\le 6\left( \intom \big(\partial_{xx}\rhost\big)^2\dd x\right)^{1/2}\le \frac12\intom \big(\partial_{xx}\rhost\big)^2\dd x + 18.
  \end{align*}
  And so,
  \begin{align}
    \label{eq:I0I3}
    I_0+I_3 \ge \frac12\Psi(t)^2\intom \big(\partial_{xx}\rhost\big)^2\dd x - 9\hat\kinet_t^s - 18.
  \end{align}
  Summation of \eqref{eq:I0I3} with \eqref{eq:I1I2} further above yields \eqref{eq:dds}.
\end{proof}
\begin{proof}[Proof of Proposition \ref{prp:regularity}]
  Recall that $\hat\kinet^s_t$ and $e^s_t$ are continuously differentiable for $s>0$ by Lemma \ref{lem:diff}
  and are continuous at $s=0^+$ by Lemma \ref{lem:continuity}.
  By minimality of $(\mrho,\mwelot)$,
  we thus obtain for each $\sigma>0$:
  \begin{align}
    \label{eq:usemin}
    0 \ge \frac{\hat\wed^\eps(\mrho,\mwelo)- \hat\wed^\eps(\rho^\sigma,\welo^\sigma)}\sigma
    \ge \fint_0^\sigma\int_0^T\left(-\frac{e^{-t/\eps}}{\eps}\frac{\dd}{\dn s}\big(\eps\hat\kinet^s_t+e_t^s\big)\right)\dd t\dd s.
  \end{align}
  Now substitute the lower bound \eqref{eq:dds} for the integrand.
  Rearranging terms and neglecting the non-negative contributions outside of the time interval $[2\eps,T-2\eps]$,
  we obtain
  \begin{equation}
    \label{eq:godhelpme}
    \begin{split}
      \fint_0^\sigma\int_{2\eps}^{T-2\eps}\intom&\left\{
        \frac12\big(\partial_{xx}\rhost\big)^2
        +\frac\eps4\intom\big(\welost\big)^2\big(\partial_x\rhost\big)^2
        + \frac{\eps\mu}8\intom\big(\partial_x\welost\big)^2
      \right\}\dd x\dd t\dd s \\
      &\le 
      \fint_0^\sigma\int_0^T\big(9 \hat\kinet_t^s + 18\big)\dd t\dd s.
    \end{split}
  \end{equation}
  By continuity of $s\mapsto \kinet_t^s$ in $L^1((0,T))$, see Lemma \ref{lem:continuity},
  and by the fundamental estimate \eqref{eq:kinetbound},
  the right-hand side above is estimated for each sufficiently small $\sigma>0$ as follows:
  \begin{align}
    \label{eq:CT}
    \fint_0^\sigma\int_0^T\big(9 \hat\kinet_t^s + 18\big)\dd t\dd s
    \le 18T+10\int_0^T\hat\kinet_t^0\dd t \le C_T:=18T + 5\nrg(\bar\rho_0).
  \end{align}
  It follows from \eqref{eq:godhelpme} that there is a sequence $s_n\searrow0$ such that
  \begin{align}
    \label{eq:itallcomesfromhere}
    \int_{2\eps}^{T-2\eps}\intom\left\{
    \frac12\big(\partial_{xx}\rho_t^{s_n}\big)^2
    +\frac\eps4\intom\big(\welo_t^{s_n}\big)^2(\partial_x\rho_t^{s_n})^2
    + \frac{\eps\mu}8\intom(\partial_x\welo_t^{s_n})^2
    \right\}\dd x\dd t
    \le C_T.
  \end{align}
  By Alaoglu's theorem, we conclude --- passing to a subsequence if necessary --- weak convergence
  of $\partial_{xx}\rho^{s_n}$, of $\partial_x\welo^{s_n}$, and of the product $\welo^{s_n}\partial_x\rho^{s_n}$
  in $L^2([2\eps,T-2\eps]\times\crc)$ as $s_n\searrow0$:
  the functions $\partial_x\rho^{s_n}$ and $\welo^{s_n}$ converge strongly
  to their respective limits $\partial_x\mrho$ and $\mwelo$ in $L^2([2\eps,T-2\eps]\times\crc)$ by construction,
  and thus converge in particular pointwise a.e.\ along a suitable subsequence of $(s_n)$.
  This allows to identify the weak limits with, respectively,
  $\partial_{xx}\mrho$, $\partial_x\mwelo$, and $\mwelo\,\partial_x\mrho$.
  By means of weak lower semi-continuity of the $L^2$-norm,
  we can pass to the limes inferior $s_n\searrow0$ in the inequality \eqref{eq:itallcomesfromhere}
  and obtain the desired bound \eqref{eq:regularity}.
\end{proof}

%%%%%%%%%%%%%%%%%%%%%%%%%%%%%%%%%%%%%%%%%%%%%%%%%%%%%%%
\subsection{Weak formulation for the limit}
%%%%%%%%%%%%%%%%%%%%%%%%%%%%%%%%%%%%%%%%%%%%%%%%%%%%%%%
%
In this section, we prove Theorem \ref{thm:chpde}.
To obtain Euler-Lagrange equations for the minimizer $(\mrho,\mwelo)$,
we proceed ``one test function at a time'':
throughout this section, let a $\xi\in C^\infty_c((0,\infty)\times\crc)$ be fixed.
For quantitative estimation on variations with that specific $\xi$,
choose an adapted time horizon $T>4$ and a temporal distance $\tau\in[\eps,1)$,
so that
\begin{align*}
  \xi_t(x) = 0 \quad \text{unless $t\in(2\tau,T-2\tau)$}.
\end{align*}
%
% \begin{proposition}
%   \label{prp:eulerlagrange_mob}
%   The equation \eqref{eq:chfinal} holds.
% \end{proposition}
% 
To derive the formulation \eqref{eq:chfinal},
we apply again the direct methods from the calculus of variations.
Now, the variation $(\rho^s,\welo^s)$ of the minimizer $(\mrho,\mwelo)$ is chosen
as solution to an auxiliary coupled system of parabolic PDEs
that explicitly depends on the test function $\xi$ and has the minimizer $(\mrho,\mwelo)$ as initial datum;
a similar perturbation has been used, e.g., in the simpler setting of the minimizing movement scheme in \cite{LMS}.
For definiteness,
let $\psi\in C^\infty_c(\Rp)$ be defined from a cut-off function $\Psi\in C^\infty_c(\Rp)$ as in \eqref{eq:psifromPsi};
notice that $\xi(t;x)=0$ unless $\Psi(t)=1$.
Define further $\Xi\in C^\infty_c(\Rnn\times\crc)$ with support in $(2\tau,T-2\tau)\times\crc$ by
\begin{align}
  \label{eq:xi2Xi}
  \Xi_t = \eps e^{t/\eps}\xi_t.
\end{align}
At each fixed $t\in(0,T)$, the PDE system is given by
\begin{align}
  \label{eq:chpert1}
  \partial_s\rho^s_t &= \partial_x\big(\mob(\rho_t^s)\,\Xi_t\big) + \delta\psi(t)\,\partial_{xx}\rho_t^s, \\
  \label{eq:chpert2}  
  \partial_s\welo^s_t &= -\partial_t\big(\mob(\rho_t^s)\,\Xi_t\big) + \delta\big(\psi(t)\,\partial_{xx}\welo_t^s-\psi'(t)\,\partial_x\rho_t^s\big),
\end{align}
and subject to the initial conditions
\begin{align}
  \label{eq:chpertic}
  \rho^0_t=\mrhot,\quad \welo^0_t=\mwelot.
\end{align}
We emphasize that in \eqref{eq:chpert1}\&\eqref{eq:chpert2},
the test function is multiplied by $\mob(\rho_t^s)$ and \emph{not} by $\mob(\rho_t^s)+\mu$.
Indeed, the degeneracy $\mob(0)=\mob(1)=0$ is essential to guarantee
the confinement $\rho_t^s(x)\in I$ for $s>0$ by the comparison principle.
\begin{lemma}
  \label{lem:continuity1}
  For almost every $t\in(0,T)$, there exist a global-in-$s$ solution $(\rho_t^s,\welo_t^s)$
  to \eqref{eq:chpert1}\&\eqref{eq:chpert2} with initial datum \eqref{eq:chpertic}.
  Specifically,
  $s\mapsto\rho_t^s$ and $s\mapsto\welo_t^s$
  are continuous from $[0,\infty)$ to $H^1(\crc)$,
  are differentiable from $(0,\infty)$ to $L^2(\crc)$,
  and attain values in $H^2(\crc)$ for $s>0$,
  so that \eqref{eq:chpert1}\&\eqref{eq:chpert2} are satisfied as equality between functions in $L^2(\crc)$.
  Consequently, $s\mapsto\hat\kinet_t^s$ and $s\mapsto e_t^s$,
  given by
  \[ \hat\kinet_t^s:= \hat\kinet\big(\rho_t^s,\welo_t^s)
    \quad\text{and}\quad
    e_t^s := \nrg(\rho_t^s),
  \]
  and continuous for $s\ge0$ and differentiable for $s>0$.
\end{lemma}
\begin{proof}
  For the proof, we rely on the theory of mild solutions for parabolic equations,
  as detailed e.g. in \cite {Henry1981geometric}.
  As sectorial operator, we use the scaled periodic Laplacian $A:=-\delta\,\partial_{xx}$ on the Hilbert space $X:=L^2(\crc)$,
  and the corresponding exponential map is given by convolution with the heat kernel $\mk$ from \eqref{eq:Gtilde},
  \[ e^{sA}u = \mk^{\delta s}\ast u \quad\text{for all $s\ge0$ and $u\in X$}. \]
  In this situation, the interpolation spaces $X^1$ and $X^{1/2}$ are isomorphic to $H^2(\crc)$ and to $H^1(\crc)$, respectively.

  % Choose $t\in(0,T)$ as a Lebesgue point of $(0,T)\ni t\mapsto\partial_t\rho_t\in L^2(\crc)$.
  To begin with, let $t\in(0,T)$ be arbitrary.
  Observe that the system  % \eqref{eq:perturbation_rho}\&\eqref{eq:perturbation_welo}
  \eqref{eq:chpert1}\&\eqref{eq:chpert2} is staggered.
  The equation \eqref{eq:chpert1} for $\rho^{(\cdot)}_t$ is autonomous (recall that $t$ is fixed).
  The nonlinearity $\rho\mapsto\partialx(\mob(\rho)\xi_t)$ is locally Lipschitz continuous
  from $X^{1/2}\cong H^1(\crc)$ to $X=L^2(\crc)$,
  and grows at most linearly as $\|\rho \|_{H^1}\to\infty$.
  By the fundamental result \cite[Theorem 3.3.3]{Henry1981geometric} on local existence and uniquess,
  there is a unique local mild solution $s\mapsto \rho^s_t\in H^1(\crc)$ for the initial datum $\rho^0_t=\mrhot$,
  i.e., $\rho_t^s$ is the (only) fixed point of
  \begin{align}
    \label{eq:henry1}
    \rho_t^s= \mk_t^{\delta s}\ast \mrho_t + \int_0^s\partial_x\big[\mk_t^{\delta(s-r)}\ast\big(\mob(\rho_t^s)\,\Xi_t\big)\big]\dd r,
  \end{align}
  and the result \cite[Corollary 3.3.5]{Henry1981geometric} on global extendability,
  that solution is even global in $s\ge0$.
  The claimed continuity and differentiability properties are all part of the mild solution concept, see \cite[Definition 3.3.1]{Henry1981geometric}.
  
  Recall from Corollary \ref{cor:wxL2} that $\mrho_t$ has a weak time derivative in $L^2$.
  Taking difference quotients in \eqref{eq:henry1} with respect to $t$,
  using the regularity of $\mob$ and of $\Xi$,
  and applying standard heat kernel estimates,
  it can be shown that for almost every $t$,
  the partial derivative $s\mapsto \partial_t\rho_t^s\in L^2(\crc)$ is H\"older continuous.
  The H\"older continuity is obviously inherited by $s\mapsto \partial_t\big(\mob(\rho_t^s)\Xi_t)$.
  This is sufficient to solve the --- linear but inhomogeneous --- auxiliary equation \eqref{eq:chpert2} for $\welo$ as well,
  see \cite[Theorem 3.2.2]{Henry1981geometric}.
  The solution has an integral representation that is analogous to \eqref{eq:henry1}:
  \begin{align}
    \label{eq:henry2}
    \welo_t^s= \mk_t^{\delta s}\ast \mwelo_t - \delta s\,\psi'(t)\,\partial_x\mk_t^{\delta s}\ast\mrho_t
    + \int_0^s\partial_t\big[\mk_t^{\delta(s-r)}\ast\big(\mob(\rho_t^s)\,\Xi_t\big)\big]\dd r.
  \end{align}
  Continuity of $s\mapsto e_t^s$ and of $s\mapsto \hat K_t^s$ follow
  from continuity of $s\mapsto\rho_t^s\in H^1(\crc)$ and of $s\mapsto\welo_t^s\in H^1(\crc)$.
\end{proof}
\begin{lemma}
  $(\rho^s,\welo^s)\in\curves(\bar\rho_0)$ for each $s\ge0$.
  Moreover, for each fixed $s>0$ and almost every $t\in(2\tau,T-2\tau)$,
  the density $\rho^s$ attains only values in the interior of $I=[0,1]$.
\end{lemma}
\begin{proof}
  On basis of the integral representations \eqref{eq:henry1}\&\eqref{eq:henry2},
  it is easily deduced that at any $s>0$, the functions $\rho^s$ and $\welo^s$
  --- which are defined almost everywhere on $(2\tau,T-2\tau)\times\crc$ ---
  satisfy the continuity equation in distributional sense:
  for each $\theta\in C^\infty_c((2\eps,T-2\eps)\times\crc)$,
  \begin{align*}
    &\int_0^T\intom\big(\rho^s_t\,\partial_t\theta_t + \welo_t^s\,\partial_x\theta_t\big)\dd x\dd t \\
    &= \int_0^T\intom\big\{\mk^{\delta s}_t\ast\mrho_t \,\partial_t\theta_t
      - \delta s\,\psi'(t)\,\partial_x\mk_t^{\delta s}\ast\mrho_t\,\partial_x\theta_t
      + \mk^{\delta s}_t\ast\mwelo_t\,\partial_x\theta_t\big\}\dd x\dd t \\
    &\quad +\int_0^T\intom\int_0^s
      \big\{\partial_x\big[\mk_t^{\delta(s-r)}\ast\big(\mob(\rho_t^s)\,\Xi_t\big)\big] \,\partial_t\theta_t
      -\partial_t\big[\mk_t^{\delta(s-r)}\ast\big(\mob(\rho_t^s)\,\Xi_t\big) \big]\partial_x\theta_t\big\}\dd r\dd x\dd t \\
    &= \int_0^T\intom\big\{\mrho_t \big(\mk^{\delta s}_t\ast(\partial_t\theta_t)+(\partial_t\mk^{\delta s}_t)\ast\theta_t) + \mwelo_t\,\partial_x(\mk^{\delta s}_t\ast\theta_t)\big\}\dd x\dd t \\
    &\quad -\int_0^s\int_0^T\intom
      \mk_t^{\delta(s-r)}\ast\big(\mob(\rho_t^s)\,\Xi_t\big) \,
      \big[\partial_x\partial_t\theta_t- \partial_t\partial_x\theta_t\big]
      \dd x\dd t\dd r \\
    &=\int_0^T\intom\big\{\mrho_t\,\partial_t\big(\mk_t^s\ast\theta_t\big) + \mwelo_t\,\partial_x\big(\mk_t^s\ast\theta_t\big)\big\}\dd x\dd t + 0 = 0.
  \end{align*}
  The regularity of $s\mapsto\rho^s_t$ for fixed $t$ --- values in $H^2(\crc)$ and $s$-derivative in $L^2(\crc)$ ---
  is sufficient to apply the comparison principle with the constant solutions $\hat\rho\equiv1$ and $\check\rho\equiv0$,
  and conclude that $\rho_t^s$ only attains values in the interior of the interval $I$.
\end{proof}
\begin{lemma}
  \label{lem:cont9}
  As $s\searrow 0$, we have that 
  \begin{align*}
    \partial_x\rho^s \to \partial_x\bar\rho^\eps,\ \welo^s\to\bar\welo^\eps \quad \text{in $L^2\big((0,T)\times\crc\big)$}.
  \end{align*}
  Consequently, $\hat\kinet^s\to\hat\kinet^0$ and $e^s\to e^0$ in $L^1((0,T))$.
\end{lemma}
\begin{proof}
  Again, this follows by direct estimates on the integral representations \eqref{eq:henry1} and \eqref{eq:henry2}.
  The estimates are tedious but fully explicit, and provide even a rate of convergence as $s\searrow0$
  if one uses that $\partial_{xx}\bar\rho$ and $\partial_x\bar\welo$ belong to $L^2([2\tau,T-2\tau]\times\crc)$
  to quantify the proximity of $\mk_t^s\ast\partial_x\mrho_t$ and $\mk_t^s\ast\mwelo_t$
  to the respective limits $\partial_x\mrho$ and $\mwelo$.
  We leave the details to the reader.
\end{proof}
\begin{lemma}
  \label{lem:diff1}
  The derivatives of $\hat\kinet_t^s$ and $e_t^s$ at $t>0$ with respect to $s>0$ satisfy
  \begin{equation}
    \label{eq:diff1}
    -\frac{e^{-t/\eps}}{\eps}\,\frac{\dd}{\dn s}\big(\eps\hat K^s_t+e^s_t\big)
    \ge  A_t^s + \delta\Psi(t)^2B_t^s - \delta(9 \hat\kinet_t^s + 18),
  \end{equation}
  with $A_t$ and $B_t$ given by
  \begin{align}
    \label{eq:Ast}
    A_t^s&=
         \intom \partial_{xx}\rhost\,\partial_x\big(\mob(\rhost)\xi_t\big)\dd x
         + \intom\frac{\mob(\rhost)}{\mob(\rhost)+\mu}\welost\,\xi_t\dd x \\
    \nonumber
         &\qquad +\eps\left\{\intom\frac{\welost}{\mob(\rhost)+\mu}\partial_t\big(\mob(\rhost)\xi_t\big)\dd x
         + \intom\frac{\mob'(\rhost)}2\left(\frac{\welost}{\mob(\rhost)+\mu}\right)^2\partial_x\big(\mob(\rhost)\xi_t\big) \dd x\right\}, \\
       % & + \eps \delta\Psi(t)^2 \left\{\intom\big(\mob(\rhost)+\mu\big)\left[\partial_x\left(\frac{\welost}{\mob(\rhost)+\mu}\right)\right]^2\dd x
       %   +\intom \frac{-\mob''(\rhost)}2\left(\frac{\welost}{\mob(\rhost)+\mu}\right)^2(\partial_x\rhost)^2\right\}\dd x \\
    \label{eq:Bst}
    B_t^s&=
           \frac12 \intom\big(\partial_{xx}\rhost\big)^2\dd x
           +\frac\eps4\intom\big(\welost\big)^2\big(\partial_x\rhost\big)^2\dd x
           + \frac{\eps\mu}8\intom\big(\partial_x\welost\big)^2\dd x.
         % &+ \delta\Psi(t)^2\intom(\partial_{xx}\rhost)^2\dd x
         % + \delta \Psi(t)\big(\Psi(t)+2\eps\Psi'(t)\big)\intom\frac{\welost}{\mob(\rhost)+\mu}\partial_x\rhost\dd x.
  \end{align}
\end{lemma}
\begin{proof}
  By differentiability of the perturbation $(\rho^s,\welo^s)$ in $H^1(\crc)$ with respect to $s>0$,
  the $s$-derivative can be interchanged with differentiation,
  and the $H^2(\crc)$-regularity is sufficient to allow a variety of integration by parts.
  To begin with,
  \begin{align*}
    -\frac{\dd}{\dn s}e_t^s
    = -\intom \partial_{x}\rhost\,\partial_x\big(\partial_s\rhost\big)\dd x
    = \intom \partial_{xx}\rhost\,\big\{\delta\psi(t)\partial_{xx}\rhost+\partial_x\big(\mob(\rhost)\Xi_t\big)\big\}\dd x,
  \end{align*}
  and so, recalling \eqref{eq:psifromPsi2} and \eqref{eq:xi2Xi},
  \begin{align}
    \label{eq:e-diff1}
    -\frac{e^{-t/\eps}}{\eps}\frac{\dd}{\dn s}e^s_t
    = \delta\Psi(t)^2\intom \big(\partial_{xx}\rhost\big)^2\dd x
    + \intom\partial_{xx}\rhost\,\partial_x\big(\mob(\rhost)\xi_t\big)\dd x.
  \end{align}
  The derivative of the kinetic term is slightly more complicated:
  \begin{align*}
    -\frac{\dd}{\dn s}\hat\kinet_t^s
    &= -\intom \left\{\frac{\welost}{\mob(\rhost)+\mu}\partial_s\welost - \frac{\mob'(\rhost)}2\left(\frac{\welost}{\mob(\rhost)+\mu}\right)^2\partial_s\rhost\right\}\dd x \\
    &= \delta\psi(t)\intom \left\{-\frac{\welost}{\mob(\rhost)+\mu}\partial_{xx}\welost + \frac{\mob'(\rhost)}2\left(\frac{\welost}{\mob(\rhost)+\mu}\right)^2\partial_{xx}\rhost\right\}\dd x \\
    &\quad + \delta\psi'(t)\intom \frac{\welost}{\mob(\rhost)+\mu}\partial_x\rhost\dd x \\
    & \quad + \intom \left\{\frac{\welost}{\mob(\rhost)+\mu}\partial_t\big(\mob(\rhost)\Xi_t\big)
    + \frac{\mob'(\rhost)}2\left(\frac{\welost}{\mob(\rhost)+\mu}\right)^2 \partial_x\big(\mob(\rhost)\Xi_t\big)\right\} \dd x .
  \end{align*}
  
  For the terms that come with a pre-factor $\delta$ --- including the one in \eqref{eq:e-diff1} ---
  we proceed exactly like in the derivation of \eqref{eq:dds}, and thus obtain the contribution $B_t^s$;
  note that the perturbation $(\rhost,\welost)$ above is not identical to the perturbation used in the context of Lemma \ref{lem:diff},
  but the integral expressions are identical, and the same manipulations can be applied. Notice that the assumption $\tau\ge\eps$ allows us to obtain the estimates in similar form as in Lemma \ref{lem:diff}, in particular, we may repeat the estimate of $\eps|\Psi'|\le1$ employed in \eqref{eq:I3}.
  % 
  % Integrate by parts in the integral with pre-factor $\delta\psi(t)$ to reduce the second derivates of $\rhost$ and $\welost$ to first derivatives.
  % The integrand then becomes
  % \begin{align*}
  %   &\partial_x\left(\frac{\welost}{\mob(\rhost)+\mu}\right)\left(\partial_{x}\welost - \frac{\welost}{\mob(\rhost)+\mu}\partial_x\mob(\rhost)\right)
  %   -\frac{\mob''(\rhost)}2\left(\frac{\welost}{\mob(\rhost)+\mu}\right)^2(\partial_{x}\rhost)^2\\
  %   &\quad = \left[\partial_x\left(\frac{\welost}{\mob(\rhost)+\mu}\right)\right]^2
  %   +\frac{-\mob''(\rhost)}2\left(\frac{\welost}{\mob(\rhost)+\mu}\right)^2(\partial_{x}\rhost)^2.
  % \end{align*}
  % 
  Concerning the remaining terms, it suffices to observe that
  \begin{align*}
    e^{-t/\eps}\partial_t\big(\mob(\rhost)\Xi_t\big)
    = \eps\partial_t\big(\mob(\rhost)\xi_t\big) + \mob(\rhost)\xi_t
  \end{align*}
  to obtain  the result \eqref{eq:diff1}.
  % 
  % we obtain in total:
  % \begin{equation}
  %   \label{eq:k-diff1}
  %   \begin{split}
  %     -e^{-t/\eps}\frac{\dd}{\dn s}\hat\kinet_t^s
  %     &= \delta\eps\Psi(t)^2\intom \left\{\left[\partial_x\left(\frac{\welost}{\mob(\rhost)+\mu}\right)\right]^2
  %       +\frac{-\mob''(\rhost)}2\left(\frac{\welost}{\mob(\rhost)+\mu}\right)^2(\partial_{x}\rhost)^2\right\}\dd x \\
  %     &\quad + \delta\Psi(t)\big(\Psi(t)+2\Psi'(t)\big)\intom \frac{\welost}{\mob(\rhost)+\mu}\partial_x\rhost\dd x \\
  %     &\quad +\eps\intom \left\{\frac{\welost}{\mob(\rhost)}\partial_t\big(\mob(\rhost)\xi_t\big)
  %       + \frac{\mob'(\rhost)}2\left(\frac{\welost}{\mob(\rhost)+\mu}\right)^2 \partial_x\big(\mob(\rhost)\Xi_t\big)\right\} \dd x \\
  %     &\quad + \intom \frac{\mob(\rhost)}{\mob(\rhost)+\mu}\welost\xi_t\dd x.
  %   \end{split}
  % \end{equation}
  % Addition of \eqref{eq:e-diff1} and \eqref{eq:k-diff1} yields the result \eqref{eq:diff1}.
\end{proof}
\begin{lemma}
  There exists a constant $C>0$ --- depending on $\eps>0$, $\mu>0$, $\delta>0$ and on $\xi$, but independent of $s>0$ and $t\in(0,T)$ ---
  such that expressions $A_t^s$ and $B_t^s$ defined in \eqref{eq:Ast} and \eqref{eq:Bst}, respectively, satisfy
  \begin{equation}
    \label{eq:diffbelow}
    A_t^s + \frac\delta2\Psi(t)^2 B_t^s \ge - C\left(1+\hat\kinet_t^s+e_t^s\right) .
  \end{equation}
\end{lemma}
\begin{proof}
  For $t$ not in $(2\tau,T-2\tau)$, the test function $\xi_t$ vanishes, and so $A_t^s=0$;
  the claim \eqref{eq:diffbelow} follows trivially.
  Assume $t\in(2\tau,T-2\tau)$ from now on, 
  and accordingly $\Psi\equiv 1$.
  We estimate the four integrals in $A_t^s$ separately,
  bearing in mind that the $\xi$ is a fixed smooth and compactly supported test function:
  \begin{align*}
    \intom \partial_{xx}\rhost\,\partial_x\big(\mob(\rhost)\xi_t\big)\dd x
    &\ge -\frac\delta4\intom \big(\partial_{xx}\rhost\big)^2\dd x - \frac4\delta\big(\|\partial_x\xi_t\|_\infty^2+\|\xi_t\|_\infty^2e_t^s\big), \\
    \intom\frac{\mob(\rhost)}{\mob(\rhost)+\mu}\welost\,\xi_t\dd x
    &\ge -\|\xi_t\|_\infty\big(1+\hat\kinet_t^s\big), \\
    \eps\intom\frac{\welost}{\mob(\rhost)+\mu}\partial_t\big(\mob(\rhost)\xi_t\big) \dd x
    &= \eps \intom \frac{\mob(\rhost)\welost}{\mob(\rhost)+\mu}\partial_t\xi_t
      + \eps \intom \frac{\mob'(\rhost)\,\welost}{\mob(\rhost)+\mu}(-\partial_x\welost) \xi_t\dd x \\
    &\quad\ge -\eps\|\partial_t\xi_t\|_\infty\big(1+\hat\kinet_t^s\big)
      - \frac{\eps\mu^2\delta}{16}\intom \frac{\big(\partial_x\welost\big)^2}{\mob(\rhost)+\mu}
      - \frac{8\eps}{\mu^2\delta}\|\xi_t\|_\infty\hat\kinet_t^s, 
  \end{align*}
  
  and finally,
  \begin{align*}
    &\eps\intom \frac{\mob'(\rhost)}2\left(\frac{\welost}{\mob(\rhost)+\mu}\right)^2\partial_x\big(\mob(\rhost)\xi_t\big)\dd x \\
    &\quad= \frac\eps2\intom\frac{\mob'(\rhost)\mob(\rhost)}{\mob(\rhost)+\mu}\frac{(\welost)^2}{\mob(\rhost)+\mu}\partial_x\xi_t\dd x
      - \frac\eps2\intom \mob'(\rhost)^2\left(\frac{\welost}{\mob(\rhost)+\mu}\right)^2\partial_x\rhost\,\xi_t\dd x \\
    &\quad\ge -\eps\|\partial_x\xi_t\|_\infty\hat\kinet_t^s -\frac{\eps}{\mu^3\delta}\|\xi_t\|_\infty^2\hat\kinet_t^s
      -\frac{\eps\mu^3\delta}8\intom\frac{\big(\welost\big)^2\big(\partial_x\rhost\big)^2}{\big(\mob(\rhost)+\mu\big)^3}\dd x.
  \end{align*}
  
  Summation with $\delta/2 B_t^s$ yields the lower bound \eqref{eq:diffbelow}.
\end{proof}
\begin{lemma}
  There is a constant $C_T$, expressible in terms of $T>4$ and the initial energy $\nrg(\bar\rho_0)$ alone,
  such that, for all $\delta\in(0,1)$:
  \begin{align}
    \label{eq:almostdone}
    \begin{split}
      &\int_0^T\intom \bigg(
        \partial_{xx}\mrho\,\partial_x\big(\mob(\mrho)\xi_t\big)
        +\frac{\mob(\mrho)}{\mob(\mrho)+\mu}\mwelo\,\xi_t \\
        &\qquad +\eps\left\{
          \frac{\mwelo}{\mob(\mrho)+\mu}\partial_t\big(\mob(\mrho)\xi_t\big)
          + \frac{\mob'(\mrho)}2\left(\frac{\mwelo}{\mob(\mrho)+\mu}\right)^2\partial_x\big(\mob(\mrho)\xi_t\big)
        \right\}
      \bigg)\dd x \dd t
      \le C_T\delta.
    \end{split}
  \end{align}
\end{lemma}
\begin{proof}
  The starting point is again the inequality \eqref{eq:usemin} that follows by minimality of $(\mrho,\mwelo)$
  in combination with continuity at $s=0^+$ and continuous differentiability for $s>0$ of $\hat\kinet_t^s$ and $e_t^s$,
  see Lemma \ref{lem:continuity1}.
  Using the lower bound \eqref{eq:diff1}, we obtain
  --- after elementary manipulations analogous to the ones performed in the proof of Proposition \ref{prp:regularity} ---
  for all sufficiently small $\sigma>0$ that
  \begin{align*}
    \fint_0^\sigma\int_{2\tau}^{T-2\tau} \big(A_t^s+\delta B_t^s\big)\dd t\dd s \le \delta C_T,
  \end{align*}
  with the constant $C_T$ from \eqref{eq:CT}.
  It follows that there exists a monotone null sequence $(s_n)$ of $s_n\in(0,1)$ with the property
  \begin{align}
    \label{eq:fromAB}
    \int_{2\tau}^{T-2\tau} \big(A_t^{s_n}+\delta B_t^{s_n}\big)\dd t \le \delta C_T.
  \end{align}
  Bounding $A_t^{s_n}+\frac\delta2B_t^{s_n}$ from below by \eqref{eq:diffbelow},
  and recalling the fundamental a priori estimates \eqref{eq:kinetbound} and \eqref{eq:potenzbound},
  we obtain
  \begin{align*}
    \frac\delta2\int_{2\tau}^{T-2\tau} B_t^{s_n} \dd t
    \le \delta C_T + C\int_0^T\big(1+\hat\kinet_t^{s_n}+e_t^{s_n}\big)\dd t
    \le \delta C_T + C(1+T)\big(1 + \nrg(\bar\rho_0)\big).
  \end{align*}
  By Alaoglu's theorem, we conclude
  that $\partial_{xx}\rho^{s_n}$, $\partial_x\welo^{s_n}$, and $\welo^{s_n}\partial_x\rho^{s_n}$
  converge weakly in $L^2([2\tau,T-2\tau]\times\crc)$, possibly after passing to a subsequence of $(s_n)$.
  Using the continuity of $\partial_x\rho^s$ and $\welo^s$ in $L^2([2\tau,T-2\tau]\times\crc)$ at $s=0^+$
  from Lemma \ref{lem:cont9},
  we extract a further subsequence of $(s_n)$
  so that $\partial_x\rho^{s_n}$ and $\welo^{s_n}$ converge pointwise a.e.,
  which allows to identify the weak limits as, respectively,
  $\partial_{xx}\mrho$, $\partial_x\mwelo$, and $\mwelo\,\partial_x\mrho$. 

  In summary, we have by Lemma \ref{lem:cont9}
  the following strong convergence in  $L^2([2\tau,T-2\tau]\times\crc)$,
  \begin{align}
    \label{eq:thesearestrong}
    \rho^{s_n}\to\mrho,\quad
    \partial_x\rho^{s_n}\to\partial_x\mrho,\quad
    \welo^{s_n}\to\mwelo,
    % \quad
    % \frac{\welo^{s_n}}{\mob(\rho^{s_n})+\mu} \to \frac{\mwelo}{\mob(\mrho)+\mu}
  \end{align}
  and by the reasoning above the following weak convergence in $L^2([2\tau,T-2\tau]\times\crc)$,
  \begin{align}
    \label{eq:theseareweak}
    \partial_{xx}\rho^{s_n}\rightharpoonup\partial_{xx}\mrho, \quad
    \partial_{x}\welo^{s_n}\rightharpoonup\partial_{x}\mwelo, \quad
    \welo^{s_n}\,\partial_x\rho^{s_n} \rightharpoonup \mwelo\,\partial_x\mrho.
  \end{align}
  Further, the following functions are bounded and converge to their respective limits almost everywhere:
  \begin{align}
    \label{eq:thesearebounded}
    \mob(\rho^{s_n})\to\mob(\mrho), \quad
    \mob'(\rho^{s_n})\to\mob'(\mrho), \quad
    \frac1{\mob(\rho^{s_n})+\mu} \to \frac1{\mob(\mrho)+\mu}
  \end{align}
  Moreover, observing that $\partial_x\mob(\rho^s)=\mob'(\rho^s)\partial_x\rho^s$,
  and that $\partial_t\mob(\rho^s)=-\mob'(\rho^s)\partial_x\welo^s$ by the continuity equation,
  we also have the strong respectively weak convergence
  \begin{align}
    \label{eq:thesearestrongweak}
    \partial_x\mob(\rho^{s_n})\to\partial_x\mob(\mrho), \quad
    \partial_t\mob(\rho^{s_n})\rightharpoonup\partial_t\mob(\mrho).    
  \end{align}
  We can now pass to the limit with the time integral of $A^s$ term by term:
  the limits of three integrals are immediate
  from \eqref{eq:theseareweak}, \eqref{eq:thesearestrong}, \eqref{eq:thesearebounded},
  and particularly from \eqref{eq:thesearestrongweak}:
  \begin{align*}
    \int_0^T\intom \partial_{xx}\rho^{s_n}\,\partial_x\big(\mob(\rho^{s_n})\xi_t\big)\dd x\dd t
    &\to \int_0^T\intom \partial_{xx}\mrho\,\partial_x\big(\mob(\mrho)\xi_t\big)\dd x\dd t, \\
    \int_0^T\intom \frac{\mob(\rho^{s_n})}{\mob(\rho^{s_n})+\mu}\welo^{s_n}\,\xi_t\dd x \dd t
    &\to\int_0^T\intom\frac{\mob(\mrho)}{\mob(\mrho)+\mu}\mwelo\,\xi_t\dd x\dd t, \\
    \int_0^T\intom\frac{\welo^{s_n}}{\mob(\rho^{s_n})+\mu}\partial_t\big(\mob(\rho^{s_n})\xi_t\big)\dd x\dd t 
    &\to \int_0^T\intom\frac{\mwelo}{\mob(\mrho)+\mu}\partial_t\big(\mob(\mrho)\xi_t\big)\dd x\dd t . 
  \end{align*}
  To obtain the limit in the last integral, we rewrite the integrand as follows:
  \begin{align*}
    &\int_0^T\intom\frac{\mob'(\rho^{s_n})}2\left(\frac{\welo^{s_n}}{\mob(\rho^{s_n})+\mu}\right)^2\partial_x\big(\mob(\rho^{s_n})\xi_t\big) \dd x\dd t \\
    &= \frac12\int_0^T\intom\left(\frac{\mob'(\rho^{s_n})}{\mob(\rho^{s_n})+\mu}\right)^2
      \welo^{s_n}\,\big(\welo^{s_n}\,\partial_x\rho^{s_n}\big)\,\xi_t \dd x\dd t %\\
      % &\qquad
      + \frac12\int_0^T\intom\frac{\mob(\rho^{s_n})\mob'(\rho^{s_n})}{\big(\mob(\rho^{s_n})+\mu\big)^2}
      \big(\welo^{s_n}\big)^2\partial_x\xi_t\dd x\dd t.
  \end{align*}
  For passage to the limit in the first integral,
  we combine strong convergence of $\welo^{s_n}$ with weak convergence of the product $\welo^{s_n}\partial_x\rho^{s_n}$,
  for the second, strong convergence of  $\welo^{s_n}$ is sufficient.

  To conclude \eqref{eq:almostdone},
  perform the limit in the inequality \eqref{eq:fromAB}, neglecting the non-negative contribution from $B^s$.
\end{proof}
% 
% \begin{lemma}
%   There is a constant $C$ --- depending on $\eps>0$ and $\xi$, but independent of $\delta>0$ ---
%   such that
%   \begin{equation}
%     \label{eq:chprefinal}
%     \begin{split}
%       &\int_0^T\intom \Big\{\partial_{xx}\bar\rho_t\,\partial_x\big(\mob(\bar\rho_t)\xi_t\big) + \frac{\mob(\bar\rho_t)}{\mob(\bar\rho_t)+\mu}\bar\welo_t\xi_t\Big\} \dd x\dd t \\
%       &+ \eps \int_0^T\intom\left\{\frac{\bar\welo_t}{\mob(\bar\rho_t)+\mu}\partial_t\big(\mob(\bar\rho_t)\xi_t\big)
%         + \frac{\mob'(\bar\rho_t)}2\left(\frac{\bar\welo_t}{\mob(\bar\rho_t)+\mu}\right)^2\partial_x\big(\mob(\bar\rho_t)\xi_t\big)\right\}\dd x\dd t \\
%       &\ge -C\delta.
%     \end{split}
%   \end{equation}
% \end{lemma}
% %
% \begin{proof}
%   Define the extensions $\hat\rho^s:(0,T)\times\crc\times(0,1)\to\R$ and $\hat\welo^s:(0,T)\times\crc\times(0,1)\to\R$ by
%   \begin{align*}
%     \hat\rho^s(t;x;\sigma) = \rho_t^{\sigma s}(x).
%   \end{align*}
%   It follows that $\hat\rho^s\to\hat\rho^0$ in $L^2((0,T)\times\crc\times(0,1)$.
%   
%   From the fact that $(\mrho,\mwelo)$ is a minimizer,
%   it follows that $\partial_{xx}\hat\rho^s$ is $s$-uniformly bounded in $L^2((0,T)\times\crc\times(0,1))$,
%   and thus converges weakly in that space to $\partial_{xx}\hat\rho^0$:
%   weak convergence along sub-subsequences of $s\searrow0$ follows by Aloglu's theorem,
%   and since the limit is uniquely determined from \eqref{}, the weak limit holds for $s\searrow0$.
% \end{proof}
%
\begin{proof}[Proof of Theorem \ref{thm:chpde}]
  Since the estimate \eqref{eq:almostdone} holds with a $\delta$-uniform constant $C_T$,
  the left-hand side is \emph{non-positive}.
  The analogous argument applies with $\xi$ replaced by $-\xi$,
  which shows that the left-hand side of \eqref{eq:almostdone} is \emph{non-negative}.
  Thus, it must be zero, which produces \eqref{eq:chfinal}.
\end{proof}

%%%%%%%%%%%%%%%%%%%%%%%%%%%%%%%%%%%%%%%%%%%%%%%%%%%%%%% 
\subsection{Limit of vanishing regularization}
%%%%%%%%%%%%%%%%%%%%%%%%%%%%%%%%%%%%%%%%%%%%%%%%%%%%%%%
%
\begin{proof}[Proof of Theorem \ref{thm:chlimit}]
  The first part of the proof is performed in complete analogy to that of Theorem \ref{thm:tflimit},
  that is, 
  the weak solution $(\bar\rho^*,\bar\welo^*)$ is obtained as limit of the minimizers $(\bar\rho^{\eps_n},\bar \welo^{\eps_n})$
  along a suitable sequence $\eps_n\searrow0$.
  Recall that we assume the connection \eqref{eq:epsmu} between $\mu$ and $\eps$,
  so that $\mu_n:=\mu(\eps_n)$ tends to zero as well. 
  
  On the one hand, the general considerations from Section \ref{sct:WED}
  provide convergence $\bar\rho^{\eps_n}_t\to\bar\rho^*_t$ in the bounded Lipschitz distance $\BL$,
  locally uniformly with respect to $t\ge0$;
  this guarantees weak continuity of $\bar\rho^*$ in time,
  and in particular attainment of the initial datum, $\bar\rho^*(0)=\bar\rho_0$.
  On the other hand,
  on grounds of the a priori estimate \eqref{eq:regularity},
  and by means of a diagonal argument with respect to time intervals,
  one concludes --- similarly as in \eqref{eq:wrapup} --- that
  \begin{align*} 
    \bar\rho^{\eps_n}\to\bar\rho^* \quad
    \text{weakly in $L^2_\loc((0,\infty);H^2(\crc))$}, \quad
    \text{strongly in $L^2_\loc((0,\infty);H^1(\crc))$}.
  \end{align*}
  In addition, and without loss of generality, we may assume that $\bar\rho^{\eps_n}$ converges pointwise a.e.
  Concerning $\bar\welo^{\eps_n}$,
  it suffices to observe that the general a priori bound \eqref{eq:kinetbound}
  implies weak convergence to a limit $\bar\welo^*$ in $L^2_\loc((0,\infty)\times\crc)$,
  and that the continuity equation ---  in distributional form --- passes to the limit,
  \begin{align*}
    \partial_t\bar\rho^* + \partial_x\bar\welo^* = 0.
  \end{align*}
  What is slighly more challenging than in the proof of Theorem \ref{thm:tflimit}
  is to obtain the weak form \eqref{eq:tf0weak_m} of the Cahn-Hilliard equation \eqref{eq:intro-ch}
  as limit in the integral formulation \eqref{eq:chfinal}.
  Let a test function $\xi:=\partial_x\Phi$ be fixed, with support in $[2\tau,T-2\tau]\times\crc$,
  and assume $\eps_n<\tau$ in the following.
  For the computations below, we omit the sub-index $n$ of $\eps_n$ and $\mu_n$.

  To begin with, we have pointwise a.e. convergence
  of the uniformly bounded quotient $\mob(\mrho)/(\mob(\mrho)+\mu)$ to the constant one.
  With weak convergence of $\bar\welo^\eps$, this implies
  \begin{align*}
    \int_0^T\intom \frac{\mob(\mrho)}{\mob(\mrho)+\mu}\mwelo\,\xi\dd x\dd t
    \to 
    \int_0^T\intom \bar\welo^*\,\partial_x\Phi \dd x\dd t
    = \int_0^T\intom \bar\rho^*\, \partial_t\Phi\dd x\dd t.
  \end{align*}
  Next, strong convergence of $\partial_x\mrho$
  implies that
  \begin{align*}
    \partial_x\big(\mob(\mrho)\xi\big) \to \partial_x\big(\mob(\bar\rho^*)\,\partial_x\Phi\big)
    \quad \text{strongly in $L^2_\loc((0,\infty)\times\crc)$},
  \end{align*}
  and so, in combination with the weak convergence of $\partial_{xx}\mrho$:
  \begin{align*}
    \int_0^T\intom \partial_{xx}\mrho\,\partial_x\big(\mob(\mrho)\xi\big)\dd x\dd t
    \to
    \int_0^T\intom \partial_{xx}\bar\rho^*\,\partial_x\big(\mob(\bar\rho^*)\,\partial_x\Phi\big)\dd x\dd t.
  \end{align*}
  It remains to verify that the contributions with prefactor $\eps$ in \eqref{eq:chfinal} vanish in the limit.
  Recalling that $\mob\le1$ and $|\mob'|\le1$, we obtain
  \begin{align*}
    &\left|\eps\int_0^T\intom \,\frac{\mwelo}{\mob(\mrho)+\mu}\partial_t\big(\mob(\mrho)\xi\big)\dd x\dd t\right| \\
    &\quad \le \eps\left(\int_0^T\intom\frac{(\mwelo)^2}{\mob(\mrho)+\mu}\dd x\dd t\right)^{1/2}
      \left(2\int_0^T\intom\frac{\big[\mob(\mrho)\,\partial_t\xi\big]^2+\big[\mob'(\mrho)\,\partial_x\mwelo\,\xi\big]^2}{\mob(\mrho)+\mu}\dd x\dd t\right)^{1/2}\\
    &\quad \le2 \left(\int_0^T\hat\kinet (\mrho,\mwelo)\dd t\right)^{1/2}\left(\eps^2\|\partial_t\xi\|_\infty^2 T+\frac{\eps}{\mu^2}\|\xi\|_\infty^2\left[\eps\mu\int_0^T\intom(\partial_x\mwelo)^2\dd x\dd t\right]\right)^{1/2},
  \end{align*}
  as well as
  \begin{align*}
    &\left|\eps\int_0^T\intom\frac{\mob'(\mrho)}2\left(\frac{\mwelo}{\mob(\mrho)+\mu}\right)^2\partial_x\big(\mob(\mrho)\xi\big)\dd x\dd t\right| \\
    &\quad \le \frac\eps2\left|\int_0^T\intom\frac{\mob(\mrho)\mob'(\mrho)(\mwelo)^2}{\big(\mob(\mrho)+\mu\big)^2}\,\partial_x\xi\dd x\dd t\right|
      + \frac\eps2\left|\int_0^T\intom \mob'(\mrho)^2\left(\frac{\mwelo}{\mob(\mrho)+\mu}\right)^2\,\partial_x\mrho\,\xi \dd x\dd t\right| \\
    &\quad \le \frac\eps{2}\|\partial_x\xi\|_\infty\int_0^T\intom\frac{(\mwelo)^2}{\mob(\mrho)+\mu}\dd x\dd t \\
    &\qquad + \frac\eps{2\mu^{3/2}}\|\xi\|_\infty\left(\int_0^T\intom\frac{(\mwelo)^2}{\mob(\mrho)+\mu}\dd x\dd t\right)^{1/2}
      \left(\mu^3\int_0^T\intom\frac{\big(\mwelo\,\partial_x\mrho\big)^2}{(\mob(\mrho)+\mu)^3}\dd x\dd t\right)^{1/2} \\
    &\quad\le \eps\|\partial_x\xi\|_\infty\int_0^T\hat\kinet (\mrho,\mwelo)\dd t
    + \frac{\eps^{1/2}}{2\mu^{3/2}}\|\xi\|_\infty\left(2\int_0^T\hat\kinet (\mrho,\mwelo)\dd t\right)^{1/2}
      \left(\eps\int_0^T\intom\big(\mwelo\,\partial_x\mrho\big)^2\dd x\dd t\right)^{1/2}.
  \end{align*}
  To conclude that the respective right-hand sides vanish as $\eps\to0$,
  it suffices to recall the general bound \eqref{eq:kinetbound} on the integrated kinetic energy,
  the $L^2$-bound on $\partial_x\mwelo$ and on the product $\mwelo\,\partial_x\mrho$ from \eqref{eq:regularity},
  and that $\mu$ has been chosen with $\eps/\mu^3\to0$ in \eqref{eq:epsmu}.
\end{proof}

%%%%%%%%%%%%%%%%%%%%%%%%%%%%%%%%%%%%%%%%%%%%%%%%%%%%%%%
%%%%%%%%%%%%%%%%%%%%%%%%%%%%%%%%%%%%%%%%%%%%%%%%%%%%%%%
\appendix
%%%%%%%%%%%%%%%%%%%%%%%%%%%%%%%%%%%%%%%%%%%%%%%%%%%%%%%
%%%%%%%%%%%%%%%%%%%%%%%%%%%%%%%%%%%%%%%%%%%%%%%%%%%%%%%

%%%%%%%%%%%%%%%%%%%%%%%%%%%%%%%%%%%%%%%%%%%%%%%%%%%%%%%
\section{An integral estimate}
%%%%%%%%%%%%%%%%%%%%%%%%%%%%%%%%%%%%%%%%%%%%%%%%%%%%%%%
%
The following is a simple estimate in the spirit of \cite{villani_inequality}.
%
%\begin{lemma}
%  \label{lem:villani}
%  Every smooth positive function $f:\crc\to\Rp$ satisfies
%  \begin{align}
%    \label{eq:villani}
%    \intom f_x^4\dd x \le \frac34\intom \big[(f^2)_{xx}\big]^2\dd x. 
%  \end{align}
%\end{lemma}
%
%\begin{proof} % [Proof of Lemma \ref{lem:villani}]
%  On the one hand,
%  \begin{align*}
%    \big[(f^2)_{xx}\big]^2 = 4\big[ff_{xx}+f_x^2\big]^2 = 4\big[f_x^4+2ff_x^2f_{xx}+f^2f_{xx}^2\big].
%  \end{align*}
%  On the other hand,
%  \begin{align*}
%    \intom ff_x^2f_{xx}\dd x = \frac13 \intom f\big((f_x)^3\big)_x\dd x
%    = -\frac13\intom f_x^4\dd x.
%  \end{align*}
%  In combination,
%  \begin{align*}
%    \intom \big[(f^2)_{xx}\big]^2\dd x
%    = 4\intom \left[\frac13f_x^4 + f^2f_{xx}^2\right]\dd x
%    \ge \frac43\intom f_x^4\dd x.
%  \end{align*}
%\end{proof}
\begin{lemma}
  \label{lem:villani}
  Every smooth positive function $f:\crc\to\Rp$ satisfies
  \begin{align}
    \label{eq:villani}
    \intom\frac{f_x^2}{f}\dd x \le 3\left(\intom f_{xx}^2\dd x\right)^{1/2}. 
  \end{align}
\end{lemma}
\begin{proof}
This follows from
   \begin{align}
    \label{eq:villani_help}
    \left(\intom\frac{f_x^4}{f^2}\dd x\right)^{1/2} \le 3\left(\intom f_{xx}^2\dd x\right)^{1/2} 
  \end{align}
by means of H\"oders's inequality. The validity of \eqref{eq:villani_help} may be shown by integrating the identity
\begin{align*}
    \left(\frac{f_x^3}{f}\right)_x=-\frac{f_x^4}{f^2}+3\frac{f_x^2}{f}f_{xx}.
\end{align*}
Indeed, applying H\"oders's inequality again, we obtain
\begin{align*}
    \intom \frac{f_x^4}{f^2}\dd x =3 \intom \frac{f_x^2}{f}f_{xx}\dd x \le 3 \left(\intom \frac{f_x^4}{f^2}\dd x \right)^{1/2}\left(\intom f_{xx}^2\dd x \right)^{1/2},
\end{align*}
which is \eqref{eq:villani_help}.
\end{proof}

%\begin{lemma}
 % \label{lem:villani}
%  Every smooth positive function $f:\crc\to\Rp$ satisfies
%  \begin{align}
%    \label{eq:villani}
%    \intom \Big[\big(\sqrt{f}\big)_{x}\Big]^4\dd x \le \frac{9}{16}\intom f_{xx}^2\dd x. 
%  \end{align}
%\end{lemma}
%\begin{proof}
%The proof is based on the inequality 
%\begin{align}
%\label{eq:villani_help}
%     \intom \frac{f_x^4}{f^2}\dd x \le 9\intom f^2_{xx}\dd x,
%\end{align}
%which follows from integrating the identity
%\begin{align*}
%    \left(\frac{f_x^3}{f}\right)_x=-\frac{f_x^4}{f^2}+3\frac{f_x^2}{f}f_{xx}.
%\end{align*}
%Indeed,
%\begin{align*}
%    \intom \frac{f_x^4}{f^2}\dd x =3 \intom \frac{f_x^2}{f}f_{xx}\dd x \le 3 \left(\intom \frac{f_x^4}{f^2}\dd x \right)^{1/2}\left(\intom f_{xx}^2\dd x \right)^{1/2}.
%\end{align*}
%Now take the square to obtain \eqref{eq:villani_help}. To conclude, observe that $(f_x)^4/f^2=16\big[\big(\sqrt{f}\big)_{x}\big]^4.$
%\end{proof}

%%%%%%%%%%%%%%%%%%%%%%%%%%%%%%%%%%%%%%%%%%%%%%%%%%%%%%%
\section{A generalized Aubin-Lions theorem}
%%%%%%%%%%%%%%%%%%%%%%%%%%%%%%%%%%%%%%%%%%%%%%%%%%%%%%% 
%
The following variant of the Aubin-Lions compactness theorem has been proven
--- in an even more general form --- in \cite{RossiSavare}, see Theorem 2 therein.
\begin{theorem}
  \label{thm:savare}
  On a convex closed subset $U$ of a Banach space $X$, let be given:
  \begin{itemize}
  \item a lower semi-continuous functional $\fnc:U\to[0,+\infty]$ with relatively compact sublevels;
  \item a lower semi-continuous functional $g:U\times U\to[0,+\infty]$
    with the property that $g(\rho,\eta)=0$ for $\rho,\eta\in U$ implies $\rho=\eta$.
  \end{itemize}
  Consider a sequence $(\rho_n)$ of curves $\rho_n:[0,T]\to U$ with the properties that
  \begin{itemize}
  \item $(\rho_n)$ is tight with respect to $\fnc$, i.e.,
    \begin{align}
      \label{eq:tight}
      \sup_n \int_0^T \fnc(\rho_n)\dd t < \infty;
    \end{align}
  \item $(\rho_n)$ is weakly integral equi-continuous with respect to $g$, i.e.,
    \begin{align}
      \label{eq:equicont}
      \lim_{\tau\searrow0}\sup_n \int_0^{T-\tau}g\big(\rho_n(t+\tau),\rho_n(t)\big)\dd t = 0.
    \end{align}
    Then, there is a sub-sequence $(\rho_{n'})\subseteq(\rho_n)$ such that $\rho_{n'}$
    converges in measure with respect to $t\in(0,T)$ to a limit $\rho_*:[0,T]\to U$.
  \end{itemize}
\end{theorem}
In our applications of Theorem \ref{thm:savare},
we choose $X$ as some subspace of $L^1(\crc)$,
the set $U$ as $X\cap\prbr{I}$, and
$g(\rho,\eta)=\BL(\rho,\eta)$.
Moreover, we use that \eqref{eq:equicont} is implied by (usual) equi-continuity.

\bibliographystyle{plain}
\bibliography{bibliography}

\end{document}